%% file: ArGr.tex
\documentclass[11pt,a4paper]{book}

\usepackage{hyperref}

\usepackage{mathptmx}
\usepackage{amsmath}
\usepackage{amssymb}
\usepackage{amsthm}
\usepackage{array}
\usepackage{enumerate}
\usepackage{graphicx}
\usepackage{lmodern}
\usepackage{mathrsfs}

\hypersetup{colorlinks,citecolor=blue,linkcolor=blue}

\newtheorem{theorem}{Theorem}[chapter]
\newtheorem{lemma}[theorem]{Lemma}
\newtheorem{prop}[theorem]{Proposition}
\newtheorem{cor}[theorem]{Corollary}

\theoremstyle{definition}
\newtheorem{definition}[theorem]{Definition}
\newtheorem{example}[theorem]{Example}
\theoremstyle{remark}
\newtheorem{rmk}[theorem]{Remark}

\numberwithin{equation}{chapter}
\numberwithin{section}{chapter}

\DeclareMathOperator{\SL}{SL}
\DeclareMathOperator{\GL}{GL}

\DeclareMathOperator{\SO}{SO}

\DeclareMathOperator{\Aut}{Aut}
\DeclareMathOperator{\End}{End}

\DeclareMathOperator{\vol}{vol}
\DeclareMathOperator{\covol}{covol}
\DeclareMathOperator{\syst}{syst}

\newcommand{\N}{\mathbb N}
\newcommand{\Z}{\mathbb Z}
\newcommand{\Q}{\mathbb Q}
\newcommand{\R}{\mathbb R}
\newcommand{\C}{\mathbb C}

\renewcommand{\sl}{\mathfrak{sl}}

\newcommand{\F}{\mathcal F}

\renewcommand{\O}{\mathcal O}

\newcommand{\tr}{\mathrm{tr}}

\newcommand{\V}{\mathbf V}

\newcommand{\g}{\mathfrak g}
\newcommand{\h}{\mathfrak h}

\newcommand{\G}{\mathbf G}

\newcommand{\del}{\partial}

\newcommand{\W}{\mathbf{W}}

\newcommand{\Gm}{\mathbf{G}_m}
\renewcommand{\H}{\mathbf{H}}

\newcommand{\Lie}{\mathrm{Lie}}
\newcommand{\Int}{\mathrm{Int}}
\newcommand{\Ad}{\mathrm{Ad}}

\newcommand{\norm}[1]{|\!|#1|\!|}

\newcommand{\cell}{\overline{\ell}}
\begin{document}
\title{Arithmetic groups \\[1cm] {\large \normalfont Lecture notes \\
version 0.6}}
\date{Fall 2014}


\author{\large Vincent Emery}



\maketitle

\tableofcontents

\include{intro}
\include{eucl-lattices}
\include{topol-groups}

\include{alg-arithm-groups}

\include{more-alg-gr}

\include{uniform-lattices}
\include{over-nb-fields}

\bibliographystyle{amsalpha}
\bibliography{../../emery-bib.bib}

\end{document}

%% file: intro.tex
\chapter*{Introduction}

Arithmetic groups are defined by generalizing the construction of the
\emph{modular group} $\SL_2(\Z)$, which is obtained by taking matrices
with integral coefficients in the larger group $\SL_2(\R)$. As a first
-- obvious -- generalization, we can give the example of $\SL_n(\Z)$: we
have just increased the dimension.

The modular group appears in several different areas of mathematics, and
it is not surprising that the same is true for arithmetic groups.
For instance, let us mention:

\begin{description}
  \item[Number theory] The modular group is used to define the
    functional equations satisfied by \emph{modular forms}, which
    play a central role in analytic number theory. Using more general
    arithmetic groups instead of $\SL_2(\Z)$, one can extend the notion
    of modular forms to the more general setting of \emph{automorphic forms}.  
  \item[Hyperbolic geometry]
    The $n$-sphere $S^n$ is a Riemannian manifold of constant positive sectional
    curvature. It is compact. For any dimension $n$, there also
    exist compact $n$-manifolds with constant \emph{negative} sectional
    curvature. Such a manifold (compact or not) is called hyperbolic.
    In low dimensions (up to $5$) there exist geometric constructions of
    compact hyperbolic manifolds. But for $n > 5$ one needs to study
    arithmetic groups to construct such $n$-manifolds $M$. Arithmetic groups
    will appear as their fundamental groups $\pi_1(M)$.
  \item[Expander graphs]
    \emph{Expander graphs} are finite graphs that despite relatively
    small number of edges have very good connectivity.
    These graphs are useful in many areas of mathematics, but their
    importance first became clear in computer science, with application
    to coding theory notably. One way to obtain expander graphs is to
    consider the Cayley graphs of the finite groups $\SL_n(\Z/m\Z)$ for 
    $n>2$. The idea of this construction -- due to Grigori Margulis -- 
    uses deep facts from the representation theory of $\SL_n(\Z)$. Using
    other arithmetic groups (or generalizations of them) and their finite quotients, 
    one can obtain many more construction of expander graphs.
    See for instance the survey by Lubotzky \cite{Lubo12}  for more information 
    about expander graphs.
\end{description}

As for $\SL_2(\Z) \subset \SL_2(\R)$, an arithmetic group $\Gamma$ is
naturally embedded in a larger group $G$ -- typically a real Lie group.
What makes arithmetic groups interesting is that in many cases $\Gamma$
is a lattice in $G$, which means that the quotient $G/\Gamma$ has finite
volume. A lattice  $\Gamma \subset G$ is called \emph{uniform} if 
the quotient $G/\Gamma$ is compact. In this course we will mainly study
arithmetic groups that are uniform lattices. 

This course is based on many parts of the lecture notes \cite{Benoist08} 
by Yves Benoist.
We refer to these notes and to \cite{WittMorr14} for more comprehensive 
texts on arithmetic groups and lattices. Lizhen Ji's book \cite{Ji08} is an
excellent survey about the different connections of arithmetic groups
with various subjects in mathematics.

%% file: eucl-lattices.tex
\chapter{Euclidean lattices}

This chapter serves as an introduction to the more general notion of
lattice that will be presented in Chapter~2; we will also prove Mahler's
compactness criterion (Theorem~\ref{thm:Mahler}), which will be
important to us in Chapter~5.

\section{Definitions and basic facts}

We consider the Euclidean $n$-space $\R^n$, with
the standard (positive definite) scalar product $\left< \cdot, \cdot
\right>$ and the associated norm $\norm{\cdot}$. 

\begin{definition}
  \label{def:eucl-lattice}
        A \emph{(Euclidean) lattice} in $\R^n$ is a subgroup $L \subset \R^n$ 
        (with respect to vector addition) of the form
      \begin{align}
        L &= \Z x_1 \oplus \cdots \oplus \Z x_n,
        \label{eq:Eucl-lattice}
      \end{align}
      where $x_1, \dots, x_n$ is a basis of $\R^n$.
\end{definition}

\begin{rmk}
  \label{rmk:eucl-lattice-equiv-def}
  Equivalently, a Euclidean lattice $L \subset \R^n$ is an abelian
  subgroup such that
  \begin{enumerate}
    \item $L$ spans $\R^n$;
    \item and $L$ has rank $n$.
  \end{enumerate}
\end{rmk}

\begin{example}
        $\Z^n$ is a Euclidean lattice in $\R^n$.
\end{example}

For $L \subset \R^n$ as in \eqref{eq:Eucl-lattice}, we define its \emph{covolume}
-- denoted $\covol(L)$ -- 
as the Lebesgue measure $\lambda(\F)$ of the  
\emph{fundamental domain} 
\begin{align}
  \F = \left\{ a_1 x_1 + \cdots + a_n x_n \in \R^n \; |\; a_i \in
  \left[0,1\right[ \right\}
  \label{eq:fund-dom-eucl}.
\end{align}
In other words, the covolume of $L$ is defined as the Euclidean volume
of the parallelotope $\F$. As is well known, this volume is given by the
square root of the determinant of the Gram matrix, that is,
\begin{align}
        \covol(L) &= \sqrt{\det(\left<x_i, x_j\right>)}. 
  \label{eq:Gramm}
\end{align}

\begin{rmk}
The definition of the covolume does not depends on a choice of the
$\Z$-basis $x_1,\dots, x_n$: if $x'_1, \dots, x'_n$ is another $\Z$-basis
of $L$, then the associated fundamental domain $\F'$ has volume
$|\det(A)|\cdot \lambda(\F)$, where $A$ is the matrix representing the
bases change $x_i \to x_i'$. But since $(x_i)$ and $(x_i')$ are both
$\Z$-bases of $L$, then $A \in \GL_n(\Z)$ and thus $\det(A) = \pm 1$.
\end{rmk}

A fundamental domain $\F$ has the property that all its translated
sets $x + \F$ with $x \in L$ are mutually disjoint.
Moreover, the union of all translated sets covers $\R^n$. 
\begin{lemma}
  \label{vol-F-max}
  Let $S \subset \R^n$ be a (measurable) set such that all $x + S$
  are mutually disjoint for $x \in L$. Then $\vol(S) \le \covol(L)$.
\end{lemma}
\begin{proof}
  Since $\bigcup_{x \in L} x + \F = \R^n$ is a disjoint cover, we have 
  \begin{align*}
    \vol(S) & = \sum_{x \in L} \vol(S \cap (x + \F))\\
    &= \sum_{x \in L} \vol((S - x) \cap \F)\\
    &= \vol\left( \bigcup_{- x \in L} (x + S) \cap \F \right)\\
    &\le \vol(\F).
  \end{align*}
  For the last equality (3rd line) we use the fact that the $x + S$ are disjoint.
\end{proof}

Recall that a topological space $X$ is called \emph{discrete} if every
point $x \in X$ is open, i.e., all singletons $\left\{ x \right\}$ are
open (and consequently every subset $A \subset X$ is open). Clearly
$\R^n$ is not discrete, whereas its subgroup $\Z^n$ (with respect to the
induced topology) is discrete.

\begin{lemma}
  \label{lemma:L-discrete}
  Any lattice $L \subset \R^n$ is discrete.
\end{lemma}
\begin{proof}
  Let $x_1,\dots, x_n$ be a basis of $L$, and $e_1,\dots, e_n$ be the
  standard basis for $\R^n$. The linear map sending $x_i$ to $e_i$ (for
  $i = 1,\dots, n$) is an homeomorphism from $\R^n$ to itself, which
  sends $L$ to $\Z^n$. Thus $L$ is discrete, because $\Z^n$ is.
\end{proof}

Using the same argument, we see that the set $L \setminus \left\{ 0
\right\}$ is closed in $\R^n$. It follows that the norm function $\norm{\cdot}$
takes a minimum on $L \setminus \left\{ 0 \right\}$.

\begin{definition}[Systole]
  For a Euclidean lattice $L$ we define its \emph{systole} as the
  (positive) real number 
  \begin{align*}
    \syst(L) &= \min_{x \in L \setminus \left\{ 0 \right\} } \norm{x}.
  \end{align*}
\end{definition}



\begin{prop}
  For any lattice $L \subset \R^n$ we have:
  \begin{align*}
    \syst(L)  \le 2 \left( \frac{\covol(L)}{\nu_n}
    \right)^\frac{1}{n},
  \end{align*}
  where $\nu_n$ is the volume of the $n$-dimensional ball of radius $1$.
  \label{prop:Herm-Mink}
\end{prop}

\begin{proof}
  Let $B_R(0)$ be an open ball around $0$ such all  translated $x +
  B_R(0)$ are mutually disjoint, and choose $R$ to be maximal for this property. 
  Then by Lemma~\ref{vol-F-max}, we have $R^n \nu_n \le \covol(L)$.
  By maximality of $R$, for any $\varepsilon > 0$ there exists $x \in L$
  such that $B_{R+\varepsilon}(0)$ and $B_{R+\varepsilon}(x)$
  intersect. By the triangular inequality, this implies $\norm{x} \le 2 (R
  + \varepsilon)$. Letting $\varepsilon \to 0$, we obtain $x \in L$
  with $\norm{x} \le 2R$. 
\end{proof}

\section{Topological description}

Using topology, one can give an equivalent definition for the notion of 
Euclidean lattice. This will motivate the consideration of lattices as
subgroups of a group $G$ more general than $\R^n$ (see the next chapter). 
First, let recall some basic elements useful to study the 
interplay between group theory and topology. 

\begin{definition}
  \label{def:top-group}
  A \emph{topological group} $G$ is a group whose underlying set is a
  topological space such that the map $f: G \times G \to G$ given 
  by $f(x, y) = x y^{-1}$ is continuous. 
\end{definition}

\begin{rmk}
    From the definition, it follows that the inverse map $x \mapsto
    x^{-1}$ is continuous, and  for any element $a \in G$ the (left)
    multiplication map $\mu_a: x \mapsto a x$ is continuous (and
    similarly for right multiplication). 
\end{rmk}

\begin{rmk}
    A subgroup $H \subset G$ of a topological group $G$ is again a
    topological group, with respect to the induced topology.
\end{rmk}

\begin{example}
  The additive group $G = (\R^n, +)$  is a topological group with
  respect to the real topology on $\R^n$: the map $f$ is given by 
  $f(x, y) = x - y$ and is obviously continuous.
\end{example}



Let $H \subset G$ be a subgroup of a topological group $G$. Not only $H$
possesses a natural topology determined by $G$, but so does the
quotient set $G/H$ as well:
\begin{definition}
  The \emph{quotient topology} on $G/H$ is defined as the finest 
  topology making the projection map $\pi: G \to G/H$ continuous. 
\end{definition}

\begin{rmk}
        \label{rmk:open-in-quotient}
It follows from the definition that a subset $A \subset G/H$ is open  
if and only the preimage $\pi^{-1}(A)$ is open in $G$.
\end{rmk}

\begin{lemma}
        \label{lemma:proj-open}
        The projection map $\pi: G \to G/H$ is an open map 
        (it sends open sets onto open sets)
\end{lemma}
\begin{proof}
        For $U \subset G$ open, we have to show that $\pi(U)$ is open.  
        But by Remark~\ref{rmk:open-in-quotient} this is the case exactly when
        $\pi^{-1}(\pi(U))$ is open.  The latter is equal to $U H = \bigcup_{h \in H} U h$,
        which is open since each $U h$ is open.
\end{proof}

  

With the definition of the quotient topology we have enough information to state 
the topological characterization of Euclidean lattices: 
\begin{prop}
  \label{prop:eucl-lattice-topol}
        A subgroup $L \subset \R^n$ is a Euclidean lattice if and only if  
        \begin{enumerate}
                \item $\R^n / L$ is compact;
                \item and $L$ is discrete.
        \end{enumerate}
\end{prop}

\begin{proof}
  We prove the easy direction of the implication: let $L \subset \R^n$
  be a Euclidean lattice. We have already seen in
  Lemma~\ref{lemma:L-discrete} that $L$ is discrete.
  Consider the closure $\overline{\F}$ of some fundamental domain for
  $L$. If $\pi: \R^n \to \R^n/L$ denotes the projection map, we can
  check directly that $\pi(\overline{\F}) = \R^n/L$.  But $\overline{\F}$
  is compact, so that $\R^n/L$ is compact.

  For the other direction we refer to \cite[Theorems 6.1 and
  6.6]{SteTall} (note that the definition of ``lattices'' from
  \cite{SteTall} allows $L$ to have rank $< n$). 
\end{proof}

\begin{rmk}
  \label{rmk:lattice-2-opposed-conditions}
  The second condition in Proposition~\ref{prop:eucl-lattice-topol}
  indicates that $L$ is relatively small compared to its ambient 
  space $\R^n$. On the other hand, the first condition shows that $L$
  still ``captures'' much of $\R^n$. These two opposite aspects also
  appear for the conditions given in
  Remark~\ref{rmk:eucl-lattice-equiv-def}.
\end{rmk}

\begin{rmk}
  \label{rmk:Eucl-lattice-torus}
  Consider the special case $n = 1$ and $L = \Z \subset \R$. Then the
  quotient $\R/\Z$ is homeomorphic to the circle $S^1$, which is compact. 
  More generally, for a lattice $L \subset \R^n$ the quotient $\R^n/L$
  is homeomorphic to the $n$-torus $S^1 \times \cdots \times S^1$.
  See \cite[Theorem 6.4]{SteTall}.
\end{rmk}

\section{The space of Euclidean lattices}

For some fixed $n$, let $X$ denote the set of lattices in $\R^n$. 
The group $\GL_n(\R)$ acts on $X$ on the left: any $g \in \GL_n(\R)$
transforms a lattice $L$ into a lattice $gL$.
Moreover, this action is transitive as any lattice $L$ can be written $L
= g L_0$ for some $g \in \GL_n(\R)$ and $L_0 = \Z^n$. The stabilizer of
$L_0$ is the subgroup $\GL_n(\Z)$ (matrices with entries in $\Z$ and
determinant $\pm 1$). This gives an identification 
\begin{align}
  X &= \GL_n(\R) / \GL_n(\Z).
  \label{eq:X-GL}
\end{align}
Using this identification we can easily put a topology on $X$, namely
the quotient topology, where the topology on $\GL_n(\R)$ is the usual
one -- induced as a subset of the real vector space $\mathrm{Mat}(n \times n,
\R)$. 

\begin{prop}
  Both functions $\covol(\cdot)$ and $\syst(\cdot)$ are continuous on
  the space $X$.
\end{prop}
\begin{proof}
 Denoting by $\pi: \GL_n(\R) \to X$ the projection map, we have that the 
 map $f = \covol \circ \pi$ is given by $f(g) = |\det(g)|$. It is obviously
 continuous on $\GL_n(\R)$. For $U \subset \R$ an open set, the preimage 
 $f^{-1}(U)$ is thus open. But $\pi$ is open
 (Lemma~\ref{lemma:proj-open}), 
 so that $\pi(f^{-1}(U)) = \covol^{-1}(U)$ also is open; this proves 
 the continuity of the covolume on $X$. 

 The proof for the systole follows the same argument, noting that $f(g) = \syst(g L_0)$ 
 is continuous on $\GL_n(\R)$.
\end{proof}

Our goal is to prove Mahler's compactness theorem (below), which will 
be an important ingredient later during the course. For this we need the following fact 
about the geometry of lattices.

\begin{lemma}
  \label{lemma:bounded-basis}
  Let $a > 1$. There exists a constant $C(n,a)$ such that any lattice $L \subset \R^n$
  with $\syst(L) > 1/a$ and $\covol(L) < a$ has
  a basis $x_1, \dots, x_n$ with $\norm{x_i} \le C(n,a)$ for $i=1,\dots,n$.
\end{lemma}

\begin{proof}
   The proof proceeds by induction. If $n=1$, the result is clear for $C(1,a) = a$. 
   Suppose the result correct for $n-1$, and let $L \subset \R^n$ be a lattice with 
   $\syst(L) > 1/ a$ and $\covol(L) < a$. Let us choose a nonzero $v \in L$ of minimal 
   length, i.e., with $\norm{v} = \syst(L)$. We consider $V = v^\perp$, the real subspace 
   orthogonal to $v$, with the orthogonal projection $p: \R^n \to V$. The projection
   $L' = p(L)$ is a lattice of $V$ with $\syst(L') \ge \frac{\sqrt{3}}{2} \syst(L)$ (exercise). 
   Thus, $\syst(L') > \frac{\sqrt{3}}{2a}$.
   Moreover, one has $\covol(L) = \norm{v} \covol(L')$ and so $\covol(L') < a^2$.

   By the induction assumption, there exists a basis $x_1', \dots, x_{n-1}'$ of $L'$ such that
   $\norm{x_i'} \le C(n-1, \frac{2}{\sqrt{3}} a^2)$. 
   By definition of $L'$, any $x_i'$ is obtained as the projection $p(x_i)$ for some $x_i \in L$. 
   Moreover, by adding an integral multiple of $v$, we can choose $x_i$ such that its orthogonal
   projection on the line $\R v$ has norm less than $\norm{v}$.
   But by Proposition~\ref{prop:Herm-Mink}, the systole $\norm{v}$ can be bounded by some constant $D(n,a)$ that 
   depends only on $n$ and $a$.  By the triangular inequality, we have
   $\norm{x_i} \le D(n,a) + C(n-1, \frac{2}{\sqrt{3}}a^2)$, and we choose this
   upper bound to be the constant $C(n,a)$. 

   We show that for $x_n = v$, the vectors $x_1,\dots, x_n$ form a
   basis of $L$. The linear independence is clear, and it remains to
   show that these vectors generates $L$ over $\Z$. Let $x \in L$, and
   write $p(x) = \sum_{i=1}^{n-1} \lambda_i x_i'$, with $\lambda_i \in
   \Z$. Then for $y = x - \sum_{i=1}^{n-1} \lambda_i x_i$  we have $p(y)
   = 0$, and thus $y$ is contained in $L \cap \R x_n$. But since
   $\syst(L) = \norm{x_n}$, one has necessarily $y = \lambda_n x_n$ for
   some $\lambda_n \in \Z$. This gives $x = \sum_{i=1}^n \lambda_i x_i$,
   with integer coefficients.
\end{proof}

Finally we state and prove the main theorem of this chapter. Recall that a subset $A \subset X$ is 
called \emph{relatively compact} if its closure $\overline{A}$ is compact.

\begin{theorem}[Mahler's compactness criterion]
  \label{thm:Mahler}
  A subset $A \subset X$ is relatively compact if and only if both
  $\covol(\cdot)$ and $\syst(\cdot)^{-1}$ are bounded on~$A$.
\end{theorem}

\begin{proof}
  Let us write $f(L) = \syst(L)^{-1}$. If $A$ is relatively compact,
  then both $f(\overline{A})$ and $\covol(\overline{A})$ are compact
  subset of $\R$, so that $f(A)$ and $\covol(A)$ are bounded. 

  For the other direction, let us suppose $f$ and $\covol$ bounded on
  $A$. Then $f$ and $\covol$ are also bounded on the closure $\overline{A}$.
  By Lemma \ref{lemma:bounded-basis}, for each $L \in \overline{A}$ we can choose 
  a basis $x_1,\dots, x_n$ with the norms $\norm{x_i}$ bounded by some constant $C$ 
  (which does not depend on $L$). Denoting $g_L \in \GL_n(\R)$ the matrix with $x_1,\dots, x_n$ 
  as column vectors, we have $\pi(g_L) = L$. Let $B = \left\{ g_L \in \GL_n(\R) \, | \, L \in \overline{A} \right\}$, 
  so that $\pi(B) = \overline{A}$. From Proposition~\ref{prop:Herm-Mink} and the hypothesis on the systole,
  we see that $|\det(g_L)| = \covol(L)$ is bounded away from $0$ on $B$. Thus, the closure $\overline{B}$
  is included in $\GL_n(\R)$. Since $B$ is bounded, $\overline{B}$ is compact. We obtain that
  the closed set $\overline{A}$ is contained in the compact set $\pi(\overline{B})$, 
  which implies that $A$ is relatively compact.
\end{proof}

\begin{rmk}
  \label{rmk:compact-Hausdorff}
 For some authors, and in the French terminology, the definition of 
 ``compact'' for a space $X$ requires that $X$ is Hausdorff (as opposed
 to the term ``quasicompact'' for $X$ that has the Heine-Borel property
 but does not need to  be Hausdorff). 
 We have not worried about this, but we will show in the next chapter
 that the quotients we have considered (namely $\R^n/L$ and
 $\GL_n(\R)/\GL_n(\Z)$) are indeed Hausdorff. See also
 Remark~\ref{rmk:Eucl-lattice-torus}.
\end{rmk}


%% file: topol-groups.tex
\chapter{Lattices in topological groups}
\label{ch:lattices-topol-gps}

In this chapter we define the notion of a lattice in a 
topological group $G$. This will generalized the notion of Euclidean
lattice seen in the previous chapter (case $G = \R^n$). 

\section{Discrete subgroups and quotients}

In this section we will suppose that the topology on the group $G$ comes
from a metric. This assumption will simplify the proofs of the results
below (they still hold under weaker assumptions).  Note that $G = \R^n$
and $G = \GL_n(\R)$ are metric space -- and all topological groups that
we are going to consider will be metric as well.  We denote by $\Gamma
\subset G$ some discrete subgroup. 

\begin{lemma}
  \label{lemma:discrete-implies-closed}
  The subset $\Gamma$ is closed in $G$. 
\end{lemma}
\begin{proof}
  Since $G$ is a metric space, a subset $A \subset G$ is closed if any sequence in
  $A$ that converges in $G$ has its limit in $A$. Let $(\gamma_n)$ a
  sequence in $\Gamma$ converging to an element $g \in G$. It follows
  that $\mu_n = \gamma_n \gamma_{n+1}^{-1}$ converges to  $1 \in \Gamma$.
  But $\Gamma$ is discrete, so that $\mu_n$ eventually stabilizes to $1$.
  Hence $\gamma_n = \gamma_{n+1}$ for $n$ large enough, and the limit
  certainly belongs to $\Gamma$.
\end{proof}

The group structure on $\Gamma$ is essential in this proof, as the
following counterexample shows.

\begin{example}
  The subset $\left\{ \frac{1}{n} \; |\; n \in \N_{>0} \right\}$ is
  discrete, but not closed in $\R$.
\end{example}

Since $G$ is a metric space, it is necessarily Hausdorff. The previous
lemma then implies the following. 
\begin{cor}
  \label{cor:quotient-Hausdorff}
  The quotient space $G/\Gamma$ is Hausdorff.
\end{cor}
\begin{proof}
  We use the following general fact: a topological space is Hausdorff 
  if and only the diagonal $\left\{ (x, x) \;|\; x \in X \right\}$ 
  is closed in $X \times X$; equivalently $\left\{ (x, y) \in X
  \times X \;|\; x \neq y \right\}$ is open.
  Consider the product map $ \phi = \pi\times \pi: G \times G \to G/\Gamma
  \times G/\Gamma$; it is continuous with respect to the product
  topologies. Moreover, since $\pi$ is open, the map $\phi$ is open as well.
  Since $\Gamma$ is closed and $G$ is a topological group, the subset
  $R = \left\{ (g, h) \;|\; g h^{-1} \in \Gamma \right\}$ is closed in
  $G \times G$.  Then $\phi( (G\times G) \setminus R)$ is open in $G/\Gamma
  \times G/\Gamma$. But this subset is exactly the set of pairs
  $(g\Gamma, h\Gamma)$ with $g \Gamma \neq h\Gamma$. Thus $G/\Gamma$ is
  Hausdorff.
\end{proof}

\begin{lemma}
  \label{lemma:discont-action}
  There exists an open neighbourhood $U \subset G$ of $1$ such that
  $U \gamma \cap U = \emptyset$ for all $\gamma \in \Gamma \setminus
  \left\{ 1 \right\}$.
\end{lemma}
\begin{proof}
  Suppose that no such neighbourhood $U$ exists. For $n \in \N_{>0}$,
  let $U_n = B_{1/n}(1)$ be the open ball of radius $1/n$ around the
  identity $1 \in G$. Then for each $n$ there exist $\gamma_n \in \Gamma
  \setminus \left\{ 1 \right\}$ such that the intersection
  $U_n \gamma_n \cap U_n$ is nonempty. Let $g_n$ be an element in this
  intersection.
  Then $g_n$ converges to $1$ when $n$ goes to infinity, and
  $g_n \gamma_n^{-1} \to 1$ as well. 
  This is only possible if $\gamma_n \to 1$, which contradicts the
  fact that $\Gamma$ is discrete.
\end{proof}

\begin{cor}
  \label{cor:proj-local-homeo}
  The projection map $\pi: G \to G/\Gamma$ is a local homeomorphism,
  that is, each $g \in G$ has an open neighbourhood $U$ such that
  $\pi|_U: U \to \pi(U)$ is a homeomorphism. 
\end{cor}
\begin{proof}
  Let $g \in G$ and let $U$ be an open neighbourhood of $1$ as in
  Lemma~\ref{lemma:discont-action}. Then $\pi: gU \to \pi(gU)$ is   
  bijective: the surjectivity is obvious and the injectivity follows
  from the particular choice of $U$. But since $\pi$ is open, this means
  that the restriction $\pi|_{gU}$ is a homeomorphism.
\end{proof}

\section{Locally compact groups and measure}


To define the notion of lattice in complete generality (in the next
section), we need a measure on the topological group $G$ that is invariant. 
We study the existence of such a measure in this section. 

\begin{definition}[Locally compact group]
  A \emph{locally compact group} $G$ is a topological group that is
  Hausdorff and locally compact: every $g \in G$ has a compact
  neighbourhood. 
\end{definition}

\begin{example}
  $\R^n$ and $\GL_n(\R)$ are locally compact. Any discrete group is
  locally compact. 
\end{example}

Let us consider  the $\sigma$-algebra $\mathcal{B}$ of \emph{Borel
sets}, that is, $\mathcal{B}$ is the $\sigma$-algebra generated by the
open subsets of $G$. A measure $\mu$ on the measure space $(G,
\mathcal{B})$ is called \emph{(left) invariant} if for any $g \in G$ and $A \in
\mathcal{B}$, one has:
\begin{align}
  \mu(g A) = \mu(A). 
  \label{eq:left-invariant}
\end{align}
A measure $\mu$ on $G$ will be called \emph{regular} if for any $A \in
\mathcal{B}$ and any $V \subset G$ open, one has:
\begin{align}
  \mu(A) &= \inf\left\{ \mu(U) \,|\, A \subset U, U \mathrm{ open } 
\right\} ; \\
\mu(V) &= \sup \left\{ \mu(K) \,|\, K \subset V, K \mathrm{ compact }
\right\}.
\end{align}

\begin{theorem}[Haar's theorem]
  \label{thm:Haar}
  Let $G$ be a locally compact group. There exists on $(G, \mathcal{B})$ 
  a nonzero regular left invariant measure $\mu$ that is finite on compact
  subset: $\mu(K) < \infty$ for any compact $K \in \mathcal{B}$.
  Moreover, up to a scalar this measure is unique: if $\mu'$ is another
  such measure, then $\mu' = c \mu$ for some $c \in \R_{>0}$.
\end{theorem}

\begin{definition}[Haar measure]
  The measure $\mu$ given by the preceding theorem is called a
  \emph{Haar measure} on $G$.
\end{definition}

\begin{example}
  The Lebesgue measure $\lambda$ is a Haar measure on $\R^n$. It is
  normalized by $\lambda([0,1]^n) = 1$.
\end{example}

\begin{example}
  The group $\R^\times = \GL_1(\R)$ is locally compact. A Haar measure
  $\mu$ on $\R^\times$ is given by:
  \begin{align*}
    \mu(E) &= \int_E \frac{1}{|x|} d\lambda(x) \,. 
  \end{align*}
\end{example}

This last example generalizes to the following.

\begin{example}
  Let $G = \GL_n(\R)$, seen as a subset of $\R^{n\times n}$ and let
  $\lambda$ be the Lebesgue measure on the latter. Then a Haar measure
  $\mu$ on $G$ is given by: 
  \begin{align*}
    \mu(E) &= \int_E \frac{1}{|\det(g)|^n} d\lambda(g) \,. 
  \end{align*}
\end{example}

\section{Lattices in locally compact groups}

Proposition \ref{prop:eucl-lattice-topol}, which gives a topological
characterization of Euclidean lattices, suggests to define a lattice in
a topological group $G$ as a discrete subgroup $\Gamma \subset G$ with
$G/\Gamma$ compact. But another important feature of any Euclidean 
lattice $L \subset \R^n$ is the existence of a fundamental domain $\F
\subset \R^n$ of finite volume and such that $L + \F  = \R^n$
(equivalently: the projection $\pi: \F \to \R^n/L$ is surjective).
It turns out that this viewpoint is more convenient to extend
to notion of lattice to more general groups.

\begin{definition}[Lattice]
  \label{def:lattice}
  Let $G$ be a locally compact group, with Haar measure $\mu$.
  A discrete subgroup $\Gamma \subset G$ is called a \emph{lattice} 
  if there exists a measurable set $\F \subset G$ such that 
  $\F \Gamma  = G$ and $\mu(\F) < \infty$. In this case we say that 
  $\Gamma$ has \emph{finite covolume}.
\end{definition}

\begin{rmk}
  \label{lattice-indep-Haar}
  By uniqueness of the Haar measure (cf. Theorem~\ref{thm:Haar}), this
  definition does not depend on the choice of $\mu$. 
\end{rmk}

\begin{rmk}
  Let $\pi: G \to G/\Gamma$ be the projection map. The condition $\F \,
  \Gamma = G$ is equivalent to say that $\pi(\F) = G/\Gamma$. 
\end{rmk}

\begin{example}
  \label{ex:Euclidean-lattice}
  A subgroup $\Gamma \subset \R^n$ is a lattice if and only if it is a
  Euclidean lattice in the sense of Defintion~\ref{def:eucl-lattice}. 
  The measurable set $\F$ can be chosen to be a fundamental domain
  determined by the choice of a basis of $\Gamma$.
\end{example}

\begin{rmk}
  We will not use this notion in this course, but a \emph{fundamental
  domain} for $\Gamma$ in $G$ could be defined as a measurable set
  $\Gamma \subset G$ such that $\F \Gamma= G$ and which is minimal for this
  property. More precisely, the covering $G = \bigcup_{\gamma \in
  \Gamma} \F \gamma$ should be a disjoint union.
\end{rmk}

The compactness of the quotient $G/\Gamma$ implies a finite covolume for
$\Gamma \subset G$, as the following proposition shows.

\begin{prop}
  \label{prop:cocompact-implies-lattice}
  If $\Gamma \subset G$ is a discrete subgroup such that $G/\Gamma$ is
  compact, then $\Gamma$ is a lattice.
\end{prop}
\begin{proof}
  For each $y \in G/\Gamma$ let us choose $x \in G$ with $\pi(x) = y$.
  We write $U_y \subset G$ for an open neighbourhood of $x$, and (since
  $G$ is locally compact) we can choose $U_y$ to be contained in a compact subset 
  and thus to have finite measure. Since $\pi$ is open, the collection of
  subsets $\pi(U_y)$ (when $y$ runs over all elements in $G/\Gamma$) in an open 
  cover of $G$. 
  By compactness of $G/\Gamma$ there exists a finite subcover
  $\pi(U_{y_1}), \dots, \pi(U_{y_n})$ that suffices to cover all $G$. 
  Then $\F = U_{y_1} \cup \dots \cup U_{y_n}$ is a measurable set with
  $\mu(\F) < \infty$, and $\pi(\F) = G/\Gamma$. 
\end{proof}

\begin{definition}[Uniform lattice]
  A lattice $\Gamma \subset G$ is called \emph{uniform} (or
  \emph{cocompact}) if the quotient $G/\Gamma$ is compact. 
\end{definition}

\begin{example}
  For $n>1$ the discrete group $\SL_n(\Z)$ is a lattice in $\SL_n(\R)$, and
  we will see latter in this course that it is {\bf not} uniform. 
\end{example}

%% file: alg-arithm-groups.tex
\chapter{Algebraic and arithmetic groups}

The examples $\Z^n \subset \R^n$ or $\GL_n(\Z) \subset \GL_n(\R)$ motivates the
idea of ``taking integral points'' in a group in order to produce discrete
subgroups. However, for a general topological group $G$ there is no such notion
of ``integral points''. We need to consider more structure on $G$ if we want to
make sense of this construction. To obtain the needed structure we 
wil have to study a bit of algebraic geometry.

\section{Algebraic sets}

In this chapter $k$ denotes a field of characteristic zero, and $K =
\overline{k}$  an algebraic closure of $k$, with inclusion 
$k \subset K$. For any $k$-vector space $\V_k$ and any field extension
$\ell/k$, we  write  $\V_\ell = \ell \otimes_k \V_k$. Then the map $v
\mapsto 1 \otimes v$ defines a canonical inclusion $\V_k \subset \V_\ell$.
We will also denote $\V = \V_K$ (if the field $k$ is clear from the
context).

\subsubsection{Polynomial functions}
Let $\V_k$ be a $k$-vector space of finite dimension $n$. A map $f : \V_k \to k$ is
called \emph{polynomial} if it can be written as a polynomial expression
(with coefficients in $k$) in terms of elements of the dual $\V_k^*$. 
We denote by $k[\V]$ the set of polynomial maps on $\V_k$. This notation
indicates that we may see these maps as functions on $\V = \V_K$. More
generally, for a field extension $\ell / k$, we write $\ell[\V] = \ell \otimes_k
k[\V]$ and see it as the ring of polynomial functions on $\V_L$, where $L =
\cell$ is the algebraic closure.

\begin{rmk}
        \label{rmk:ring-polyn-concrete}
        This can be made more concrete by
        choosing a basis $x_1,\dots, x_n \in \V_k$. If $T_1,\dots, T_n
        \in \V_k^*$
        is its dual basis, then $k[\V]$ can be identified with the ring of
        polynomials $k[T_1, \dots, T_n]$, and then $\ell[\V] =
        \ell[T_1,\dots,T_n]$.
\end{rmk}

\subsubsection{Algebraic sets}
Any subset  $S \subset K[\V]$ defines an \emph{algebraic set} $\V(S) \subset \V$,
given by
  \begin{align}
    \V(S) = \left\{ v \in \V \, | \, f(v) = 0 \quad \forall\, f \in S  \right\}.
  \end{align}
In particular, $\V = \V(\{0\})$ is an algebraic set.

\begin{rmk}
        Note that taking zeros reverses the inclusion: if $S' \subset S$, then
        $\V(S) \subset \V(S')$.
\end{rmk}

It is not difficult to show that the algebraic sets of $\V$ form the
collection of closed subsets of a topology on $\V$, called the
\emph{Zariski topology}. This justifies the following definition.

\begin{definition}
  A subset $X \subset \V$ is called \emph{$k$-closed} if $X = \V(S)$
  for some $S \subset k[\V]$. 
\end{definition}

\begin{rmk}
       The $k$-closed subsets also define a topology on $\V$, called the
       \emph{$k$-Zariski topology}. 
\end{rmk}

\begin{rmk}
        The Zariski topology (and the $k$-Zariski topology) is {\bf not} Hausdorff.
\end{rmk}

\begin{definition}
        On any algebraic set $X \subset \V$, we call the \emph{Zariski topology} the
        topology induced by the Zariski topology on $\V$.
\end{definition}

\subsubsection{Rational points}


Let $\ell/k$ be a field extension such that $\ell \subset K$.
If $X = \V(S)$ is a $k$-closed set, 
we denote by $X(\ell)$ the set $\V_\ell \cap X$, and call it the set of
{\em $\ell$-rational points} of $X$. In particular, $X = X(K)$. And if
$X = \V$, then $X(\ell) = \V_\ell$.


Suppose that $k = \R$ (so that $K = \C$). Then $\V_\R$ (and $\V = \V_\C$), as a
real (resp.\ complex) vector space, has a natural metric topology. We usually refer to
it as the \emph{Hausdorff topology}, in order to distinguish it from the Zariski
topology on $\V$.

\begin{prop}
        \label{prop:R-points-closed}
        If $X = \V(S) $ is $\R$-closed, then $X(\R)$ is closed
        in $\V_\R$ with respect to the Hausdorff topology. Similarly, for $X$ as
        a subset of the $\C$-vector space $\V_\C$. 
\end{prop}
\begin{proof}
        This follows directly from the fact that any $f \in S$ is polynomial and
        thus continuous as a real (or complex) function.
\end{proof}

\begin{example}
  \label{eg:circle-hyperb}
        For $\V = \C^2$ one has $\C[\V] \cong \C[x,y]$. 
        Then for $f = x^2 + y^2 -1$ and $X = \V(\left\{ f \right\})$, we have
        that $X(\R)$ corresponds to  the unit circle in $\R^2$. 
        A small modification also provides a noncompact example: for $Y = \V(x^2 -
        y^2 - 1)$ we have that $Y(\R)$ is noncompact. 
\end{example}

\subsubsection{Regular functions and ideals}
For a $k$-closed subset $X \subset \V$, we denote by $I(X) \subset K[\V]$ the 
ideal defined by
\begin{align}
        I(X) = \left\{ f \in K[\V] \;|\; f(v) = 0 \quad \forall \, v \in X  \right\}. 
        \label{eq:ideal-IX}
\end{align}
Using this ideal, we define the \emph{ring of regular functions} on $X$ to be
the $K$-algebra $K[X] = K[\V]/I(X)$. The elements in $K[X]$ can indeed be
identified as (polynomials) functions on $X$. We also define:
\begin{align}
        I_k(X) &= k[\V] \cap I(X) \mbox{ and }\\ k[X] &= k[\V]/I_k(X). 
        \label{eq:kX}
\end{align}

\begin{rmk}
        \label{rmk:defined-over-k$}
Under our assumption that $\mathrm{char}(k) = 0$ (or more generally if $k$ is
\emph{perfect}), for any $k$-closed subset $X \subset \V$ the ideal $I_k(X)$ generates
$I(X)$ over $K$ (one says that $X$ is  \emph{defined over $k$}; for $k$ non-perfect
this is a stronger condition than being $k$-closed). Then one has $K \otimes_k
k[X] \cong K[X]$. 
\end{rmk}

\begin{rmk}
        \label{rmk:f-on-k-points}
        If $f \in I_k(X)$ then it certainly vanishes on $X(k)$. But the converse
        is not true: in general $I_k(X)$ is smaller than the set of functions $f
        \in k[\V]$ vanishing on $X(k)$. 
\end{rmk}

Let $X = \V(S)$ be an algebraic set. It is easy to check that $\V(S) =
\V(\sqrt{(S)})$, where $\sqrt{(S)}$ is the {\em radical} of the ideal $(S)$ generated
by $S$, which is  defined by 
\begin{align}
        \sqrt{(S)} = \left\{ f \in K[\V] \;|\; f^j \in (S) \mbox{ for some
        } j  \in \N_{>0} \right\}.
\end{align}
In other words, $\sqrt{(S)} \subset I(X)$. The opposite inclusion holds. For a
proof, see for instance \cite[\S 1.1]{Humph75}.

\begin{theorem}[Hilbert's Nullstellensatz]
        For $X = \V(S)$, one has $I(X) = \sqrt{(S)}$.
        \label{thm:nullstellen}
\end{theorem}
\begin{rmk}
        The fact that $K$ is algebraically closed is important here.
\end{rmk}

\begin{cor}
        \label{cor:closed-from-ideals}
        For $X \subset \V$ closed, one has $X = \V(I(X))$.
\end{cor}
\begin{proof}
        This follows immediately from the easy fact (noted above) that $S$ and
        $\sqrt{(S)}$ have same sets of zeros: $X = \V(S) = \V(\sqrt{(S)})$
\end{proof}

It also makes sense to define $I(A) \subset K[\V]$ -- the set of polynomials
vanishing on $A$ -- for any subset $A \subset \V$, not necessarily closed.  
Then we have the following.

\begin{lemma}
        \label{lemma:Zariski-closure}
        Let $A$ be a subset of $\V$. Then its Zariski closure $\overline{A}$ is
        equal to $\V(I(A))$.
\end{lemma}
\begin{proof}
 By definition, $\overline{A}$ is the smallest algebraic set (with respect to inclusion) 
 in $\V$ that contains $A$. Since $A \subset \V(I(A))$ and the latter is
 algebraic, we have that $\overline{A} \subset \V(I(A))$. For the other
 inclusion, note that $I(\overline{A}) \subset I(A)$, so that $\V(I(A)) \subset
 \V(I(\overline{A})) = \overline{A}$ (using
 Corollary~\ref{cor:closed-from-ideals} for the last equality).
\end{proof}




\subsubsection{Morphisms}

\label{sec:morphisms}

\begin{definition}[Morphism]
        \label{def:morphism}
       A map $\phi: X \to Y$  between two algebraic sets $X \subset \V$ and $Y
       \subset \W$ is called a \emph{morphism} if for any $f \in K[Y]$ one has:
       $f \circ \phi \in K[X]$. If both $X$ and $Y$ are defined over $k$, it is
       called a \emph{$k$-morphism} if $f \in k[Y]$ implies $f \circ \phi \in
       k[X]$.
\end{definition}

\begin{rmk}
       This gives also a natural notion of isomorphism. Two
       algebraic sets $X$ and $Y$ are isomorphic if there exists two morphisms
       $\phi: X \to Y$ and $\psi: Y \to X$ such 
       that $\phi \circ \psi = \mathrm{id}|_X$ and $\psi \circ \phi =
       \mathrm{id}|_Y$. 
       Similarly one defines the notion of $k$-isomorphism.
\end{rmk}

\begin{rmk}
        It follows from the definition that morphisms are continuous with respect to the
        Zariski topology. However, this property does not suffice to characterize
        morphisms between algebraic sets. This indicates that algebraic sets are
        more than sets with a particular topology; what is really more important
        is the sort of functions that one considers on them (here the
        polynomial functions). 
\end{rmk}

\begin{prop}
        Let $\phi$ be a $k$-morphism between the $k$-closed sets $X$ and $Y$.
        Then $\phi$ preserves the $k$-points: $\phi(X(k)) \subset Y(k)$.
\end{prop}
\begin{proof}
  For $Y \subset \W$   $k$-closed, let $T_1, \dots, T_m$ be a basis of $\W_k^*$. 
  Then for each $i = 1,\dots, m$ we have $T_i \circ \phi \in k[X]$, and
  thus for $x \in X(k)$ and $y = \phi(x)$ we have: $T_i(y) \in k$.
  This implies that $y \in \W_k$. 
\end{proof}

\begin{example}
   Let consider again the $\R$-closed subsets $X$ and $Y$ from
   Example~\ref{eg:circle-hyperb}. Then there is no $\R$-isomorphism
   from $X$ and $Y$, as this would contradict the fact that $X(\R)$ and $Y(\R)$ 
   are not homeomorphic with respect to the Hausdorff topology.
  However, these sets become isomorphic over $\C$: there exists a
   $\C$-isomor\-phism $\phi: X \to Y$. It is given by the change of
   variable $y \mapsto i y$.
\end{example}

\subsubsection{Extension of scalars}

Let $\ell/k$ be a field extension. We consider an algebraic closure
$L = \cell$  that contains $K = \overline{k}$. Of course, if
$\ell/k$ is algebraic, then $L = K$. Any $f \in k[\V]$ can be seen
as an element in $\ell[\V] = \ell[\V_L]$ and it follows that any $k$-closed set
$X  = \V_K(S)$ determines an $\ell$-closed subset $\V_L(S)$. This set is
said to obtain by \emph{extension of scalars}, and we use the notation
$X/\ell$ for it. Thus for any extension $\ell/k$, the notation $X(\ell)$
makes sense. 

\begin{rmk}
  \label{rmk:I-X-extends}
  It follows from Corollary~\ref{cor:closed-from-ideals} that $X/\ell =  \V_L(I_K(X))$.
\end{rmk}

\begin{lemma}
  Let $L$ be an algebraic closed field containing $K$, and $X \subset \V_K$
  be an algebraic set (where $K$ is algebraically closed). Then $X(K)$ is Zariski-dense in
  $X(L)$. 
        \label{lemma:Zariski-dense}
\end{lemma}
\begin{proof}
        From Lemma~\ref{lemma:Zariski-closure} we have that $\overline{X(K)} =
        \V_L(I_L(X(K)))$. We want to show that $\V_L(I_L(X(K))) = X(L)$, the
        inclusion ``$\subset$'' being obvious. To prove ``$\supset$'', we 
        have to show that any $f \in L[\V]$ that vanishes on $X(K)$
  vanishes on the whole $X(L)$.
  Let $\left\{ \alpha_i \right\}$ be a basis for $L$ over $K$.
  Then any $f \in L[\V] = L \otimes_K K[\V]$ can be written as a finite sum
  $f = \sum_i \alpha_i f_i$ with $f_i \in K[\V]$. If $f$ vanishes on
  $X(K)$, it follows that for any $i$: $f_i(v) = 0$ $\forall v \in
  X(K)$. Thus $f_i \in I_K(X)$, and by definition of $X/L$ we obtain
  that every $f_i$ (and thus $f$) vanishes on $X(L)$. 
\end{proof}

To put things in perspective: the main situation we have in mind is
$k=\Q$ and $\ell = \R$. Then, for $\overline{\Q} \subset \C$ the field
of algebraic numbers we have that $X(\overline{\Q})$ is Zariski-dense in
$X(\C)$, for any algebraic set $X$ defined over $\Q$ (or even over
$\overline{\Q}$).

Note that working over algebraically closed fields is important in
Lem\-ma~\ref{lemma:Zariski-dense}, as the following counterexample shows.
\begin{example}
       Let $\V_\Q = \Q$.  Then for $X = \V(x^2 + 1)$ we have $X(\Q) =
       \emptyset$, which is certainly not dense in $X = \left\{ \pm i \right\}$. This example also serves
       as an illustration of Remark~\ref{rmk:f-on-k-points}, since trivially
       every polynomial $f \in k[\V]$ vanishes on $X(\Q)$.
\end{example}

\subsubsection{Principal open subsets}

Let $f \in k[\V]$ be nonzero. We define the \emph{principal open} subset
(defined over $k$) $D(f) \subset \V$ to be the complement 
$\V \setminus \V(\left\{ f \right\})$. 
It is indeed open in the Zariski topology. Such a subset of $\V$ will be
considered in a way similar as an algebraic set over $k$. Again, the
inclusion in $\V$ induces a Zariski topology on $D(f)$.
The ring of regular functions on $D(f)$ is defined to be the $K$-algebra
$K[\V]$ in which we invert $f$:
\begin{align}
        K[D(f)] &=   \left\{ \frac{g}{f^s} \;|\; g \in K[\V], s \in \Z  \right\}.
\end{align}
One defines similarly $k[D(f)]$ by inverting $f$ in $k[\V]$. 
Any closed subset $X \subset D(f)$ also comes equipped with a ring of
regular functions, defined as $K[X] = K[D(f)]/I(X)$, where $I(X)$ is the
ideal in $D(f)$ of functions vanishing on $X$.

\begin{definition}[Affine variety]
        \label{def:principal-open}
       We will called a closed subset $X \subset D(f)$ in some principal  
       open subset $D(f) \subset \V$ (with $f \in K[\V]$)
       an \emph{affine variety}. We say that it is
       \emph{defined over $k$} if it is $k$-closed in $D(f)$ and $f \in k[\V]$. 
\end{definition}

\begin{rmk}
        \label{rmk:algebraic-affine-var}
Note that since $\V = D(c)$ for any nonzero constant polynomial $c \in K[\V]$,
we have that algebraic sets are affine varieties in the sense of this definition.
\end{rmk}

Using the ring of regular functions $K[X]$ (and $k[X]$), one can define
similarly to Definition~\ref{def:morphism} the notion of morphism (resp.\
$k$-morphism) between affine varieties.
\begin{prop}
       Every affine variety over $k$ is isomorphic to an algebraic set over $k$. 
\end{prop}
\begin{proof}
        The idea is the following.
        Let $D(f)$ be principal open in $\V$. Let $\W_k = \V_k \oplus k$, so that
        any element in $\W$ can be represented as a couple $(x, t)$ with $x \in
        \V$ and $t \in K$. We denote by $T$ an element in the dual $k^*$ (thus
        $T(t) = t$). Let $X \subset \W$ be the $k$-closed set $\W(f \cdot T -
        1)$. Then the map $\phi: D(f) \to X$, $x \mapsto (x, \frac{1}{f(x)})$ is a
        $k$-isomorphism, with inverse $\phi^{-1}: (x, t) \mapsto x$.
        It follows that closed subsets in $D(f)$ are isomorphic to closed subset
        in $X$, which are closed in $\W$.
\end{proof}

This proposition shows that we could have defined affine varieties simply as
algebraic sets. However, when dealing with algebraic groups it is natural to
consider principal open subset as well, as we will see in the next section.


\begin{rmk}
        \label{rmk:variety}
        We have given here an ad hoc definition of affine varieties, that fits  
        well the purpose of studying linear algebraic groups. A first goal when 
        studying algebraic geometry is to give a more intrinsic notion of
        ``variety'', which does not limit to affine varieties.   
        Projective varieties are the first examples of nonaffine varieties that
        come to mind. 
\end{rmk}

\section{Algebraic groups}

Let $V_k$ be a vector space over $k$ of dimension $n$, and denote $\W_k =
\End(\V_k)$ its ring of linear endomorphisms. It is a $k$-vector space
of dimension $n^2$, and we can consider 
the ring of regular functions $k[\W]$. The determinant $\det: \W_k \to
k$ is a polynomial function, and thus the \emph{general linear} group  
$\GL(\V)$ is a principal open subset (defined over $k$) in $\W = \W_K$,
namely corresponding to $D(\det) \subset \W$. 
Fixing a basis $x_1,\dots, x_n$ of $\V_k$, we obtain a basis
$T_{11}, \dots, T_{nn} \in \W_k^*$.  With this identification, we 
have 
\begin{align}
  k[\GL(\V)] &= k\left[T_{11},\dots, T_{nn},
    \frac{1}{\det(T_{ij})}\right].
  \label{eq:regul-funct-GL}
\end{align}

\begin{definition}[Algebraic group]
  A subgroup $\G \subset \GL(\V)$ that is $k$-closed (as a subset of
  the principal open set $\GL(\V)$)  is called a \emph{(linear) algebraic group}
  over $k$, or simply a \emph{$k$-group}.
\end{definition}

\begin{example}
  Trivially, the general linear group $\GL(\V)$ is an algebraic
  $k$-group. When $\V_k = k$ has dimension one, we write $\Gm$ (or
  $\Gm/k$ if we want to emphasize the field of definition) for
  $\GL(\V)$, and we call it the \emph{multiplicative group} (over $k$).  
\end{example}

\begin{rmk}
        \label{rmk:k-points-GL}
        Note that the $k$-points  $\GL(\V)(k)$ are given by $\GL(\V_k)$.
\end{rmk}

\begin{example}
  The \emph{special linear} group $\SL(\V) = \W(\det - 1)$ is a
  $k$-group.
\end{example}

\begin{example}
  Let $x \in \GL(\V_k)$. The subset of automorphisms of $\V$ commuting 
  with $x$:
  \begin{align*}
    \left\{ g \in \GL(\V) \; |\; gx = xg  \right\},
  \end{align*}
  is an algebraic $k$-group. To see that it is algebraic, one can choose
  a basis $T_{11}, \dots, T_{nn}$ for $\End(\V_k)$. The matrix multiplication
  by $x$ is given coordinate-wise by polynomial function with
  coefficients in $k$.  
\end{example}

An important feature of any $k$-group $\G$ is that for any field extension
$\ell/k$, the rational points $\G(\ell)$ form a group. This can be
easily check on each of the examples above. 
For the proof in general, we first note the following lemma.

\begin{lemma}
  Let $A \subset \GL(\V)$ be a subgroup (not necessarily closed). Then
  its Zariski closure $\overline{A} \subset \GL(\V)$ is a subgroup (in
  particular it is an algebraic group).
  \label{lemma:closure-is-group}
\end{lemma}
\begin{proof}
        Trivially, $1 \in \overline{A}$. Let $\phi: \GL(\V) \to \GL(\V)$ be the
        map given by $\phi(x) = x^{-1}$. It is an isomorphism of algebraic
        varieties, and in particular it is a homeomorphism for the Zariski
        topology. Then $\phi(\overline{A}) = \overline{\phi(A)} = \overline{A}$
        (since $\phi(A) = A$), which proves that $\overline{A}$ is closed under
        inversion. One proves similarly that $x \overline{A} = \overline{A}$ for
        any $x \in A$. Put a bit differently: $A y \subset \overline{A}$ for any
        $y \in \overline{A}$. But then $\overline{A y} = \overline{A} y$ must be
        contained in $\overline{A}$.
\end{proof}

\begin{prop}
  Let $\G$ be a $k$-group, and  $\ell/k$ be any field extension. Then
  $\G(\ell)$ is a group. In particular, the extension of scalars
  $\G / \ell$ is again an algebraic group. 
\end{prop}
\begin{proof}
        The result is clear when $\ell \subset K = \overline{k}$, since in this
        case $\G(\ell) = \GL(\V_\ell) \cap \G(K)$.
        In the general case, it suffices to prove that $\G(L)$ is a group for $L
        = \overline{\ell}$ and then to apply the same argument. By
        Lemma~\ref{lemma:Zariski-dense}, $\G(K)$ is Zariski-dense in $\G(L)$
        and thus the result follows directly from
        Lemma~\ref{lemma:closure-is-group}.
 \end{proof}

\begin{cor}
        Let $\G$ be an algebraic $\Q$-group. Then $\G(\R)$ is a locally compact group.
\end{cor}
\begin{proof}
 The proposition shows that $\G(\R)$ is indeed a group. As a subset of
 $\End(\V_\R)$ it is Hausdorff, and locally compact since it is closed
 in $\GL(\V_\R) \cong \GL_n(\R)$ (cf. Proposition~\ref{prop:R-points-closed}). 
\end{proof}

\begin{rmk}
        Of course the result remains true for any $\R$-group.
\end{rmk}

\begin{rmk}
  The same argument shows that $\G(\C)$ is also a locally compact group. 
  And so is $\G(\Q_p)$, where $\Q_p$ denotes the field of $p$-adic
  numbers (for $p \in \N$ a prime).
\end{rmk}

\begin{definition}[Morphism of groups]
  \label{def:morphism-group}
  Let $\G \subset \GL(\V)$ and $\H \subset \GL(\W)$ be two algebraic
  $k$-groups. A \emph{$k$-morphism (of groups)} $\phi: \G \to \H$ 
  is a group homomorphism that is also a $k$-morphism of affine varieties. 
\end{definition}

\begin{prop}
        \label{prop:ker-k-group}
        Let $\phi: \G \to \H$ be a $k$-morphism of groups.
        Then the kernel $\ker(\phi)$ is a $k$-group.
\end{prop}
\begin{proof}
       Let $e \in \H$ be the neutral element.
       The singletons ${e}$ is a $k$-closed subset of $H$, and since $\phi$
       is continuous with respect to the $k$-Zariski topology, it follows that
       $\ker(\phi) = \phi^{-1}({e})$ is $k$-closed in $\G$.
\end{proof}

\begin{example}
  The inclusion $\SL(\V) \subset \GL(\V)$ is a $k$-morphism. 
\end{example}

\begin{example}
 The determinant map $\det: \GL(\V) \to \Gm$ is a $k$-morphism of
 groups. Its kernel $\ker(\det)$ equals $\SL(\V)$.
\end{example}

A more surprising result is the fact that the image of a morphism is a closed subgroup.
See \cite[\S 7.4]{Humph75} for the proof.

\begin{prop}
        \label{prop:image-closed}
       Let $\phi: \G \to \H$ be a $k$-morphism. Then $\phi(\G)$ is a
       $k$-subgroup of $\H$. 
\end{prop}

Note that this result has no analogue in the context of topological groups, as
the following counterexample shows.
\begin{example}
 Let $T = (\R/\Z)^2$ be the two-dimensional torus. Then for any irrational
 $\alpha \in \R \setminus \Q$, the map $\phi: \R \to T$ defined by $\phi(x) =
 (x, \alpha x)$ has a non-closed image in $T$.
\end{example}

We close this section with a remark concerning the generality of our notion of
``algebraic group''. 

\begin{rmk}
        \label{rmk:affine-vs-linear}
        We have limited our discussion to \emph{linear} algebraic groups, that is,  
        closed subgroups of $\GL(\V)$. Another point of view --  more
        general -- is to define an algebraic group as a variety together with a
        group structure ``compatible'' with the structure of variety. This has
        the advantage to make sense also for variety that are not necessarily
        affine. As notable examples of projective algebraic groups we can
        mention elliptic curves.  However, it turns out that every
        \emph{affine}
        algebraic group is isomorphic to a linear algebraic group
        (see~\cite[\S~8.6]{Humph75}).
\end{rmk}

\begin{example}[Additive group]
        \label{add-group}
       Let $\V_k = k$ and $\V = K \otimes_k \V_k$. The  
       addition on $\V$ makes it an algebraic group defined over $k$.
       It can be identified as a subgroup of $\GL(K^2)$ via the isomorphism
       \begin{align}
               x \mapsto \left( \begin{array}[h]{cc}
                       1 & x \\
                       0 & 1
               \end{array}\right).
       \end{align}
       The group $(K, +)$ is referred as to the \emph{additive group}, and
       denoted by $\G_\mathrm{a}$.
\end{example}
\begin{rmk}
       More generally, one shows that for $\V$ of dimension $n$, 
       the group $(\V, +)$ is an algebraic group, 
       isomorphic to the $n$-factors product $\G_\mathrm{a} \times \dots \times \G_\mathrm{a}$. 
       From the preceding example, the latter can be embedded in
       $\GL(\W^{\oplus n})$, with $\W \cong K^2$.
\end{rmk}

\section{Arithmetic groups}

Let $\G \subset \GL(\V)$ be an algebraic $\Q$-group. We would like to produce
discrete subgroups in $\G(\R)$. To this end we need to fix some $\Z$-lattice in
$\V_\Q$.

\begin{definition}[$\Z$-lattice]
        Let $\V_\Q$ be a $\Q$-vector space of dimension $n$. A \emph{$\Z$-lattice} 
        $L \subset \V_\Q$ is a subgroup of rank $n$ (as a $\Z$-module) such that
        $\Q \otimes L = \V_\Q$ (i.e., $L$ spans $\V_\Q$ over $\Q$).
\end{definition}
In particular the $\Z$-lattice $L$ is a Euclidean lattice in $\V_\R \cong \R^n$
in the sense of Chapter 1. When the context is clear,
we will often just say ``lattice'' instead ``$\Z$-lattice''.

\begin{rmk}
        \label{rmk:generalization-lattices}
        In Chapter 2 we have generalized the notion of Euclidean lattice to
        extend it to the context of subgroup of more general topological groups
        than $G = \R^n$. The definition of $\Z$-lattice suggests another type of
        generalization -- more algebraic.
\end{rmk}

\begin{lemma}
       Let $L$ and $L'$ be two $\Z$-lattices in $\V_\Q$. Then there exists $m
       \in \Z$ such that $m L \subset L' \subset \frac{1}{m} L$. 
        \label{lemma:Z-lattices-commens}
\end{lemma}
\begin{proof}
       Let $x_1, \dots, x_n$ be a $\Z$-basis of $L$, and $x'_1,\dots, x_n'$ a
       $\Z$-basis of $L'$. Both are $\Q$-bases for $\V_\Q$. In particular, any
       $x_i'$ can be written as a $\Q$-linear combination 
       \begin{align*}
               x_i' &= \sum_{j=1}^n \alpha_{ij} x_j.
       \end{align*}
       Let $\alpha$ be the product of all denominators of  the $\alpha_{ij} \in
       \Q$. Then for every $i = 1,\dots, n$ we have $\alpha x_i' \in L$, so that
       that $\alpha L' \subset L$. Similarly, one proves the existence of
       $\beta \in \Z$ such that $\beta L \subset L'$, and the result follows for
       $m = \alpha \beta$.
\end{proof}

Note that $\GL(\V_\Q)$ (and thus any of its subgroups) acts naturally on $\V_\Q$. 
Given a lattice $L \subset \V_\Q$ and an algebraic $\Q$-group $\G \subset
\GL(\V)$ we can define the stabilizer subgroup 
\begin{align}
        \G_L = \left\{ g \in \G(\Q) \;|\; g L = L \right\}.
        \label{eq:stabil-lattice}
\end{align}
It follows that any $g \in \G_L$ acts as an automorphism of the
free $\Z$-module $L$. In particular, $\det(g) = \pm 1$.

\begin{example}
        For $\G = \GL(\V)$, we have that $\G_L$ is the full automorphism group
        $\GL(L)$. In particular, fixing a $\Z$-basis of $L$ we obtain an
        identification $\GL(L) = \GL_n(\Z)$. It follows that $\GL(L)$ is
        discrete in $\GL(\V_\R) \cong \GL_n(\R)$. 
\end{example}

\begin{prop}
        \label{prop:arithm-implies-discrete}
        The group $\G_L$ is discrete in $\G(\R)$. 
\end{prop}
\begin{proof}
       Since $\G_L \subset \GL(L)$, this follows from the discreteness of
       $\GL(L)$ in $\GL(\V_\R)$. 
\end{proof}

\begin{definition}[commensurable groups]
        \label{def:commens}
       Let $\Gamma_1$ and $\Gamma_2$ be two subgroups of a larger group $G$. 
       We say that $\Gamma_1$ is \emph{commensurable with} $\Gamma_2$  if 
       $\Delta = \Gamma_1 \cap \Gamma_2$ has finite index in both $\Gamma_1$ and
       $\Gamma_2$.
\end{definition}
\begin{rmk}
        \label{rmk:commens-equivl-rel} 
        ``To be commensurable with'' is an equivalence relation for the subgroups in
        $G$.
\end{rmk}

\begin{example}
        From Lemma~\ref{lemma:Z-lattices-commens} one easily deduces that any
        two $\Z$-lattices in the same $\Q$-vector space $\V_\Q$ must be
        commensurable (note that $m L \subset L$ has finite index: $[L : mL] =
        m^n$).
\end{example}

\begin{prop}
        \label{prop:arithm-commensur}
       Let $L$ and $L'$ be two $\Z$-lattices in the same vector space $\V_\Q$,
       and let $\G \subset \GL(\V)$ be an algebraic $\Q$-group.
       Then $\G_{L}$ and $\G_{L'}$ are commensurable in $\G(\Q)$. 
\end{prop}
\begin{proof}
        Note that for any $\alpha \in \Q^\times$ and any $\Z$-lattice $L$, we
        have $\G_{\alpha L} = \G_L$. From Lemma~\ref{lemma:Z-lattices-commens}
        we have the existence of $m \in \Z$ such that $m L \subset L' \subset
        \frac{1}{m} L$, from which we deduces that $\G_L$ acts on the set 
        \begin{align*}
                X = \left\{  M \mbox{ $\Z$-lattice in } \V_\Q \;|\; mL \subset M
                \subset \frac{1}{m} L \right\}.
        \end{align*}
        Under this action, the stabilizer of $L' \in X$ is exactly $\G_L \cap \G_{L'}$.
        Thus $\G_L/(\G_{L'} \cap \G_L)$ is in bijection with the orbit of $L'$ in
        $X$. But $X$ is a finite set (one shows that the number a sublattices of
        fixed index in a fixed lattice is finite). Similarly, one proves that
        the index $[\G_{L'}: \G_L \cap \G_{L'}]$ is finite, so that $\G_L$ and
        $\G_{L'}$ are commensurable. 
\end{proof}

This motivates that following definition.

\begin{definition}[Arithmetic subgroup]
        \label{def:arithm-sbgp}
        Let $\G \subset \GL(\V)$ be an algebraic $\Q$-group.
        Any subgroup $\Gamma \subset \G(\Q)$ commensurable with $\G_L$ for some
        $\Z$-lattice $L \subset \V_\Q$ is called an \emph{arithmetic subgroup} of
        $\G(\Q)$.
\end{definition}

The important fact to bear in mind is that as single algebraic $\Q$-group really
determines a whole commensurability class of subgroups in $\G(\Q)$.

\begin{rmk}
       The definition could be extended in 
       an obvious way to subgroups of $\G(\R)$ or $\G(\C)$. 
       But then ``arithmetic subgroups of $\G(\R)$'' could refer to more than
       one commensurability class. 
       To avoid some confusion, we choose here to limit the use of 
       the term  ``arithmetic subgroup'' for subgroups of $\G(\Q)$, and will 
       extend the notion when speaking of ``arithmetic lattices'' (next
       section).
\end{rmk}



\begin{example}
        $\GL_n(\Z)$ is an arithmetic subgroup of $\GL_n(\Q)$. Similarly for
        $\SL_n(\Z) \subset \SL_n(\Q)$.
\end{example}

\begin{example}
       For an integer $m \ge 2$, let $\phi_m : \SL_n(\Z) \to \SL_n(\Z/m)$ be the
       homomorphism given on each matrix entry by the reduction mod $m$ map $\Z
       \to \Z/m$. Its kernel $\Gamma(m) = \ker(\phi_m)$ is called a
       \emph{principal congruence subgroup} of $\SL_n(\Z)$, and since
       $\SL_n(\Z/m)$ is finite we have that $\Gamma(m)$ has finite index in
       $\SL_n(\Z)$. Thus $\Gamma(m)$ is an arithmetic subgroup of $\SL_n(\Q)$. 
\end{example}

When choosing a $\Z$-basis for $L$ we obtain an identification $L = \Z^n$, which
in turn identifies $\G_L$ with a subgroup of $\GL_n(\Z)$. We have: $\G_L
= \G(\Q) \cap \GL_n(\Z)$, and  this
justifies the notation $\G_\Z$ for $\G_L$. If $L$ is clear from the context, or
if the statement is independent of the choice of $L$, we will use this notation
$\G_\Z$.

\begin{rmk}
        \label{rmk:Z-points-caution}
  The notation $\G(\Z)$ is also commonly used.
  Both notation $\G(\Z)$ and $\G_\Z$  are a bit dangerous, since the ``$\Z$-points'' are
  not intrinsically defined in $\G$. If $\phi: \G \to \G'$ is a $\Q$-isomorphism,
  then we can only say that  $\phi(\G_\Z)$ is commensurable with $\G'_\Z$. On
  the other hand, $\phi$ induces an isomorphism $\G(\Q) \cong \G'(\Q)$.
\end{rmk}

\section{Arithmetic lattices}


\begin{prop}
        \label{prop:commens-preserves-lattice} 
        Let $\Gamma$ and $\Gamma'$ be  two subgroups of the locally compact
        group $G$, and suppose that $\Gamma$ and $\Gamma'$ are commensurable.
        Then $\Gamma$ is a lattice in $G$ if and only $\Gamma'$ is
        a lattice in $G$. Moreover, in this case $\Gamma'$ is uniform if and only if
        $\Gamma'$ is uniform.
\end{prop}
\begin{proof}
   We only have to prove the result assuming $\Gamma' \subset
   \Gamma$ with finite index. If $\Gamma$ is discrete, then
   certainly $\Gamma'$ is. We write $\Gamma = g_1 \Gamma' \cup \dots
   \cup g_n \Gamma'$ (finite union), with $g_i \in \Gamma$. If
   $\Gamma'$ is discrete, then each $g_i \Gamma'$ is discrete. It
   easily follows that $\Gamma$ is discrete. 
   Suppose that $\Gamma$ is a lattice: there exists $\F$ of finite
   measure such that $\F \, \Gamma = G$. Then for $\F' = \F g_1 \cup \dots
   \cup \F g_n$ one has $\F'\, \Gamma' = G$, so that $\Gamma'$ is also a
   lattice. The converse is obvious: if $\Gamma'$ is a lattice then
   $\Gamma$ is. We leave it as an exercise the proof concerning the
   ``uniform'' property.
\end{proof}

Thus taking commensurable subgroups preserves the quality of being a
lattice (resp.\ being a uniform lattice). It follows that if one
arithmetic subgroup of $\G(\Q)$ is a lattice in $\G(\R)$, then the same
holds for  all arithmetic subgroups of $\G(\Q)$. This shows two
important things:
\begin{itemize}
  \item The sentence ``$\G_\Z$ is a lattice (resp.\ uniform lattice)''
    in $\G(\R)$ makes sense without specifying a particular $\Z$-lattice $L$
    such that $\G_L = \G_\Z$.
  \item We can try to find condition on the $\Q$-group $\G$ that implies
    that $\G_\Z$ is a (uniform) lattice, rather than conditions on
    $\G_\Z$ itself.
\end{itemize}

It is easily shown that any arithmetic subgroup in $\Gm(\Q) \cong \Q^*$ is not
a lattice in $\Gm(\R) \cong \R^*$. This can be use to show that more generally
if the $\Q$-group is such that there exists a nontrivial
$\Q$-morphism of group $\G \to \Gm$, then $\G_\Z \subset \G(\R)$ is 
not a lattice.

\begin{definition}[Character]
        Let $\G$ be an algebraic $k$-group.
        A \emph{$k$-character} (or \emph{character over $k$}) of $\G$
        is a $k$-morphism of group $\chi: \G \to \Gm/k$. 
\end{definition}

There is always the trivial character, given by $\chi(g) = 1$ for all $g \in \G$. 

\begin{example}
       The determinant $\det: \GL(\V) \to \Gm$ is a nontrivial character. Thus,
       $\GL_n(\Z)$ is not a lattice in $\GL_n(\R)$.
\end{example}

The following fundamental theorem asserts that $\G$ having a nontrivial $\Q$-character 
is the only obstruction for $\G_\Z$ to be a lattice in $\G(\R)$. Its proof
relies on the so-called \emph{reduction theory}, and can be found in
\cite{Bor69}.

\begin{theorem}[Borel and Harish-Chandra]
       Let $\G$ be an algebraic $\Q$-group. Then $\G_\Z$ is a
       lattice in $\G(\R)$ if and only if $\G$ has no nontrivial
       $\Q$-character.
        \label{thm:Borel-HC}
\end{theorem}

In addition to Proposition~\ref{prop:commens-preserves-lattice},
some other elementary modifications preserve lattices. The
following is particularly relevant for us.

\begin{prop}
        \label{prop:lattice-compact-kern}
       Let $\phi: G \to H$ be a continuous open surjective homomorphism between
       locally compact groups, and suppose that $\phi$ has compact
       kernel. Let $\Gamma \subset G$ be a lattice (resp.\ uniform lattice).
       Then $\phi(\Gamma)$ is a lattice (resp.\ uniform lattice) in $H$.
\end{prop}
\begin{proof}
        We refer to \cite[Ch.1 Cor.~4.10]{OniVinb-I}
        for the general proof, and we prove here only the restricted case
        of $\Gamma$ uniform, assuming that $G$ and $H$ are metric.
        
        Let  show that $\phi(\Gamma)$ is discrete assuming that $G$ and
        $H$ are metric. We denote by $K \subset G$ the kernel of $\phi$. 
        Suppose that $\phi(\Gamma)$ is not discrete in
        $H$. Then there exists a sequence $g_n \in \Gamma$ such that
        $\phi(g_n) \to 1$, with all $\phi(g_n) \neq 1$. This implies
        that the distance $d(g_n K, K) > 0$ converges to $0$
        (if not: there would exist an open
        neighbourhood $U \subset G$ of $K$ such that $U \cap g_n K =
        \emptyset$  for infinitely many $g_n$ in the sequence,
        so that $\phi(U)$ would be an open neighbourhood of
        $1$ in $H$ that does not contain all but finitely of the
        $\phi(g_n)$; this contradicts the
        convergence of this sequence). Thus there exists a sequence
         $\alpha_n, \beta_n \in K$ such that
        $g_n \alpha_n \beta_n^{-1} \to 1$. Moreover, since $K$ is
        compact, by considering subsequences we may assume that
        $\alpha_n \beta_n^{-1}$ converges, say to $h^{-1} \in K$ so that
        $g_n \to h$. But since $\Gamma$ is discrete we have that
        $g_n = h$ for $n$ large enough, which implies that
        $\phi(g_n) = 1$ for $n$ large enough. This contradicts the
        assumption $\phi(g_n) \neq 1$.

        Using a similar argument as in
        the proof of Proposition~\ref{prop:cocompact-implies-lattice}, one sees
        that $G/\Gamma$ is compact if and only if there exists $\F \subset G$
        compact such that $\F \Gamma = G$. In this case $\F' = \phi(\F)$ is compact
        and such that $\F' \phi(\Gamma) = \phi(G) = H$, proving that $\phi(\Gamma)$ is
        uniform. 
\end{proof}

\begin{definition}[Arithmetic lattice]
        \label{def:arithm-lattice}
       Let $H$ be a locally compact group. A subgroup $\Gamma \subset H$ is
       called an \emph{arithmetic lattice} if there is an algebraic $\Q$-group $\G$ such that:
       \begin{itemize}
               \item $\G$ has no nontrivial $\Q$-character;
               \item there exists a surjective continuous open homomorphism $\phi: \G(\R) \to
                       H$ with compact kernel;
               \item $\Gamma$ is commensurable with $\phi(\G_\Z)$.
             \end{itemize}
\end{definition}

\begin{rmk}
        It follows from the various results stated above that $\Gamma \subset H$
        is indeed a lattice.
\end{rmk}

\begin{rmk}
As for arithmetic subgroups, a choice of a suitable $(\G, \phi)$ determines a
single commensurability class of lattices in $H$. But we can now let $\G$ and/or
$\phi$ vary, to get different commensurability classes of lattices in the same
$H$. In particular, if $\H$ is an algebraic $\Q$-group,
then any $\Q$-group $\G$ without nontrivial $\Q$-character such that
$\G(\R) \cong \H(\R)$ determines a commensurability classes of arithmetic
lattices, a priori not commensurable with $\H_\Z$. 
\end{rmk}

\begin{rmk}
        The existence of $(\G, \phi)$ imposes strong conditions on the locally
        compact group $H$. It is typically a real Lie group.
\end{rmk}

\section*{About references}

The three standard references on linear algebraic groups -- all with the same
title -- are the books \cite{Humph75} by Humphreys, \cite{Bor91} by Borel,
and \cite{Spring98} by Springer. All of them start by introducing some of the needed
theory from algebraic geometry. Humphrey's book is somewhat easier for a
beginner; the drawback is that it does not develop the ``rational'' theory -- over
a non-algebraically closed field (but it surveys this part in a last
chapter).

In most textbooks on algebraic geometry, algebraic sets are defined in the
canonical vector space $K^n$, rather than in some $n$-dimensional $\V = \V_K$ as we did. 
The former approach is easier, since it directly identifies $k[\V]$ with the
polynomial ring $k[T_1,\dots, T_n]$. Here we have chosen to work with $\V$
essentially to emphasize the non-intrinsic nature of the ``$\Z$-points'' in
$\V_\Q$. We believe that  this facilitates the discussion of arithmetic subgroups. 

Arithmetic groups and lattices are discussed for instance in
\cite{Bor69} and \cite{OniVinb-I}.

%% file: more-alg-gr.tex
\chapter{More on algebraic groups}
\label{ch:more-alg-gr}

In order to prove some results about arithmetic groups, one needs to
better understand the structure theory of algebraic groups, in which
they live. It is the goal of this chapter to survey some useful facts
concerning algebraic groups. 

We still assume that the field $k$ has characteristic zero, and $K =
\overline{k}$.

\section{Rational representations}

As it was defined in Chapter 3, any morphism of group determines a linear 
action on a vector space of finite dimension. We emphasize this action in 
the following definition.  

\begin{definition}[Rational representation]
 Let $\G$ be a $k$-group and $\V_k$ be a finite-dimensional vector space
 over $k$. A {\em $k$-rational representation} (or simply {\em
 $k$-representation}) of $\G$ on $\V$ is a $k$-morphism of group $\G \to \GL(\V)$. 
\end{definition}

\begin{rmk}
For any field extension $\ell/k$ a $k$-representation furnishes an
action of $\G(\ell)$ on $\V_\ell$.
\end{rmk}

It is also useful to consider linear actions of $\G$ on some infinite
dimensional vector spaces. This is for instance the case of $K[\G]$, seen as a
$K$-vector space. For every $x \in \G$ consider the right multiplication
map $\phi_x: g \mapsto g x$.  It is a morphism of variety from $\G$ onto itself,
and thus for any $f \in K[\G]$ the \emph{right translation} $\rho_x f = f \circ
\phi_x$ is again an element of $K[\G]$. Then, $x \mapsto \rho_x$ 
defines a representation of $\G$ on
the infinite dimensional vector space $K[\G]$. By definition, $(\rho_x f)(y) = f(yx)$.
For $x \in \G(k)$ we also obtain a restricted map $\rho_x: k[\G] \to k[\G]$. 

\begin{lemma}
  Let $\H \subset \G$ be a closed subgroup, and $I = I(\H) \subset
  K[\G]$ be the ideal of regular functions vanishing on $\H$.
  Then $\H = \left\{ x \in \G \;|\; \rho_x(I) \subset I \right\}$.
  \label{lemma:stab-rho-subgroup}
\end{lemma}

\begin{proof}
 Let $x \in \H$ and $f \in I$. Then $(\rho_x f)(y) = f(y x) = 0$ for any
 $y \in \H$, so that $\rho_x f \in I$. Conversely, suppose that
 $\rho_x(I) \subset I$. In particular for any $f \in I$, $\rho_x f$
 vanishes on the neutral element $e$: $(\rho_x f)(e) = f(x) = 0$, and
 thus $x \in \H$.
\end{proof}




The infinite dimensional representation $\rho: \G \to \GL(K[\G])$ can be
decomposed into rational subrepresentations: the following result follows for
instance from \cite[Prop.~2.3.6]{Spring98}.

\begin{prop}
        \label{prop:action-Km}
        There exists a filtration $K^0[\G] \subset k^1[\G] \subset \dots \subset
        k[\G]$ such that:
        \begin{itemize}
                \item Each $k^m[\G]$ is a finite dimensional subspace of
                        $k[\G]$.
                \item $\sum_{m \ge 0} k^m[\G] = k[\G]$.
                \item For $K^m[\G] = K \otimes_k k^m[\G]$ and $x \in \G$, we
                        have $\rho_x(K^m[\G]) \subset K^m[\G]$, and $\rho: \G
                        \to \GL(K^m[\G])$ is a $k$-representation.
        \end{itemize}
\end{prop}

\begin{theorem}[Chevalley]
  \label{thm:Chevalley}
  Let $\G$ be a $k$-group 
  and let $\H \subset \G$ be a closed $k$-subgroup.
  There exists a $k$-representation of $\G$ on a vector space $\V$
  such that $\H = \left\{ g \in \G \;|\; g \W = \W \right\}$, where
  $\W_k \subset \V_k$ is some $k$-subspace with $\dim(\W_k) = 1$. 
\end{theorem}

\begin{proof}
  Based on the above discussion we explain the idea of the proof,
  and refer to \cite[\S 11.2]{Humph75} for the details.
  Let $I_k(\H) \subset k[\G]$ be the ideal of functions 
  vanishing on $\H$. Then (since $\H$ is defined over $k$): $K \otimes_k I_k =
  I(\H) \subset K[\G]$. Moreover, it is know that $k[\G]$ (and $K[\G]$) is a
  \emph{Noetherian ring}, meaning that all its ideals are finitely generated.
  This implies that there exists an integer $m$ such that 
  \begin{align*}
          I^m_k &= k^m[\G] \cap I_k(\H)
  \end{align*}
  generates the ideal $I_k(\H)$; in other words,
  the subset $\H$ corresponds exactly to the set
  of points $x \in \G$ such that $f(x) = 0$ for all $f \in I_k^{m}$.
  Let $I^{m} = K\otimes_k I_k^{m} = K^m[\G] \cap I(\H)$.

  By Lemma~\ref{lemma:stab-rho-subgroup} and
  Proposition~\ref{prop:action-Km}, an element $x \in \G$ belongs 
  to $\H$ if and only if $\rho_x(I^{m}) \subset I^{m}$.
  Let $p = \dim(I^{m})$, and define the $k$-vector space
  $\V_k$ to be the exterior algebra $\bigwedge^p k^m[\G]$. 
  Let further define $\W_k$ to be the one-dimensional subspace
  $\W_k = \bigwedge^p I_k^{m}$. It follows that for this $\W \subset \V$ the
  rational representation of $\G$ on $\V$ naturally induced by the
  $k$-representaton $\rho$  has the desired properties.  
\end{proof}

\begin{rmk}
       Chevalley's theorem is also useful for putting
       a natural structure of algebraic variety on $\G/\H$,
       by identifying this quotient as a orbit in the projective space
       $\mathbb{P}(\V)$ of $\V$.
\end{rmk}

\begin{cor}
  \label{cor:Chevalley}
  Suppose moreover that the subgroup  $\H \subset \G$ has no nontrivial $k$-character. 
  Then there exists $v \in \V_k$ such that $\H$ is the stabilizer of $v$:
  \begin{align*}
    \H = \left\{ g \in \G \;|\; g v = v \right\}.
  \end{align*}
\end{cor}
\begin{proof}
  Since $\dim(\W) = 1$, the $k$-morphism $\H \to \GL(\W)$ corresponds to
  a $k$-character, and thus the action of $\H$ on $\W$ is trivial. It
  suffices then to pick any $v \in \W_k$. 
\end{proof}

\section{The Lie algebra}

We would like to define a notion of tangent space for an algebraic $k$-group
$\G$; we will consider such a tangent space at the neutral element $e = 1 \in \G$. 
Let first consider the case $\G = \GL(\V)$ for some $k$-vector space $\V_k \cong
k^n$.  Then $\GL(\V)$ is open in the vector space $\End(\V) \cong K^{n\times
n}$, and the ``analytical situation'' $\GL_n(\R) \subset \R^{n \times n}$
suggests to define the tangent of $\GL(\V)$ (say at $e=1$) as the whole vector
space $\End(\V)$. 

Let us choose a basis $(T_{ij})$ for $\End(\V_k)^*$, so that any $f \in
K[\GL(\V)]$ is expressible as a polynomial in $T_{ij}$ and $\det(T_{ij})^{-1}$.
In particular for any $T_{ij}$ the {\bf formal} partial differential
$\frac{\del}{\del T_{ij}} f$ makes sense, and is again an element in
$K[\GL(\V)]$. Then  the total differential is defined at any $x \in
\GL(\V)$ by $ df_x = \sum \frac{\del f}{\del T_{ij}} (x) \cdot T_{ij} \in \End(\V)^*$.

Let $\G \subset \GL(\V)$ be an algebraic $k$-group. Then we define its tangent
space at $e = 1$ as the vector subspace 
\begin{align}
        \Lie(\G) &= \left\{ X \in \End(\V) \;|\; df_e(X) = 0 \; \forall f \in I(\G)
        \right\};
        \label{eq:Lie-algebra}
\end{align}
and we call it the \emph{Lie algebra} of $\G$. It is customary to denote
the Lie algebra by the corresponding minuscule Fraktur letter: 
$\g = \Lie(G), \mathfrak{h} = \Lie(\H)$, etc.
We further define $\g_k = \g \cap \End(\V_k)$, and more generally in the same
way $\g_\ell$ for any field extension $\ell/k$.

\begin{example}
        The Lie algebra $\mathfrak{gl}(V)$ of $\GL(\V)$ is the whole
        vector space $\End(\V)$. 
\end{example}

\begin{rmk}
        More precisely, the term ``Lie algebra'' refer to the vector space $\g =
        \Lie(\G)$ together with the anticommutative product $[X, Y] = XY - YX$,
        which is called the \emph{Lie bracket}. One can check that this 
        product (which is defined here on the level of $\End(\V)$) is indeed
        closed inside $\g$, that is,   $[\cdot, \cdot]: \g \times \g \to \g$. 
        In other words: $\g$ is a Lie subalgebra of $\mathfrak{gl}(V)$.
\end{rmk}

\begin{rmk}
        We could use the Lie algebra to define a notion of dimension for $\G$,
        by setting $\dim(\G) = \dim(\g)$. The usual way to proceed it to define
        $\dim(\G)$ using commutative algebra -- as the \emph{Krull dimension} of
        $K[\G]$. One then checks that it coincides with $\dim(\g)$, which
        reflects the fact that $\G$ is ``smooth at $e$'' (and  thus smooth
        everywhere).
\end{rmk}

\begin{example}
        The Lie algebra $\mathfrak{sl}(V)$ of $\SL(\V)$ is given by
        \begin{align*}
                \mathfrak{sl}(V) &= \left\{ X \in \mathfrak{gl}(\V) \;|\; \mathrm{tr}(X) =
        0 \right\}. 
        \end{align*}
        Thus, $\dim(\SL(\V)) = \dim(\sl(\V)) = n^2 -1$, where $n = \dim(\V)$.
        The fact that the Lie bracket is closed in $\mathfrak{sl}(V)$ follows
        from the equality: $\tr(XY) = \tr(YX)$.
\end{example}

If $\phi: \G \to \H$ is a $k$-morphism of groups, then we obtain a natural
notion of differential at $e$: $d\phi_e: \g \to \Lie(H)$. It is the algebraic
version of the usual ``analytic'' differential that is defined for $k = \R$ or
$\C$. More precisely,  if $k \subset \C$ we can always consider the extension of
scalars  $\G/\C$, and for $X \in \g_\C \subset \End(\V_\C)$ the differential can be then
computed by the usual formula:
\begin{align}
        d\phi_e(X) &= \lim_{t \to 0} \frac{\phi(e + t X) - \phi(e)}{t}.
        \label{eq:diff-C}
\end{align}

\begin{rmk}
     That $d \phi_e (X)$ indeed lies in $\Lie(\H)$ can be easily checked by using
     the chain rule $d(f \circ \phi)_e = df_{\phi(e)} \circ d\phi_e$, still
     valid in the algebraic setting (here for $f \in K[\H]$). 
     From our definitions it is more difficult to check -- but still true --
     that for a $k$-morphism the differential restricts to a $k$-linear map
     $\g_k \to \Lie(\H)_k$. Similarly, one has $[\g_k , \g_k] \subset \g_k$, so that
     the Lie bracket restricts to the $k$-points (this justifies the term ``Lie
     algebra'' for $\g_k$ as well).
\end{rmk}

For $k \subset \R$, recall the exponential map $\exp: \GL(\V_\R) \to
\End(\V_\R)$ defined by the (converging) sequence:
\begin{align}
        \exp(X) = e^X &= \sum_{j=0}^\infty \frac{X^j}{j!}.
        \label{eq:exponent-map}
\end{align}
Then for any algebraic group $\G \subset \GL(\V)$ over $k \subset \R$, 
the exponential restricts to a map $\exp: \g_\R \to \G(\R)$. 

\begin{rmk}
        The group $\G(\R)$ is an example of a \emph{Lie group}, and much of its
        study is based on the study of the real Lie algebra $\g_\R$.
\end{rmk}

\begin{example}
       In the case of $\SL$, the fact that $\exp(\sl(\V_\R)) \subset \SL(\V_\R)$  
       follows from the formula $e^{\tr(X)} = \det(e^X)$.
\end{example}

\section{The adjoint representation}

For the rest of this section we suppose that $k \subset \R$.
Let $\G$ be a $k$-group, and for $g \in \G$ we consider the \emph{inner
automorphism} $\Int(g): \G \to \G$, defined  by $\Int(g) x = g x g^{-1}$. 
The differential $d\Int(g)_e$ determines a linear map from $\g \to \g$, and thus an
element of $\GL(\g)$. We denote by $\Ad(g): \g \to \g$ this map.

\begin{lemma}
       Let $\G \subset \GL(\V)$, and $g \in \G$.  Then for $X \in \g$, one has
       $\Ad(g) X = g X g^{-1}$.
        \label{lemma:Adj}
\end{lemma}
\begin{proof}
        Since we assume that $k \subset \R$, this can be proved simply be
        computing the differentiation formula \eqref{eq:diff-C}, where $X \in
        \End(\V)$ is seen as an element of $\C^{n\times n}$ (for $n = \dim(\V)$).
\end{proof}

The explicit formula for $\Ad(g)$ given in the lemma permits to show the
following important facts:

\begin{itemize}
        \item If $g \in \G(k)$, then $\Ad(g)$ restricts to a $k$-linear map
                $\g_k \to \g_k$. Similarly for any field extension $\ell/k$.
        \item $\Ad(g)$ preserves the Lie bracket $[\cdot, \cdot]$ on $\g$, and thus  
                is a automorphism of Lie algebra: $\Ad(g) \in \Aut(\g)$.
        \item The map $\Ad: \G \to \GL(\g)$ is a $k$-morphism, and thus a
                $k$-representation of $\G$ on $\g$. It follows from
                Proposition~\ref{prop:image-closed} that the image $\Ad(\G)$ is
                a $k$-subgroup of $\GL(\g)$.
\end{itemize}

\begin{definition}
        The representation $\Ad: \G \to \GL(\g)$ is called the \emph{adjoint
        representation} of $\G$. The image $\Ad(\G)$ is called the
        \emph{adjoint} group of $\G$.
\end{definition}

\begin{rmk}
        One can show that the group of Lie algebra automorphisms $\Aut(\g)$ is
        itself a $k$-group (i.e., a $k$-closed subgroup of $\GL(\g)$). Then
        $\Ad(\G)$ is a $k$-subgroup of $\Aut(\g)$.
\end{rmk}

For a $k$-group $\G$, let us denote by $Z(\G)$ its center (it is an algebraic
group). If $g \in Z(\G)$ obviously $\Int(g) : \G \to \G$ is the identity, so
that $\Ad(g): \g \to \g$ is trivial. Thus $Z(\G) \subset \ker(\Ad)$. In
characteristic zero, the converse holds as well, assuming that $\G$ is connected
(i.e., connected with respect to the Zariski topology). See \cite[\S 13.4]{Humph75}
for the proof.

\begin{prop}
        \label{prop:kern-Ad-center}
        If $\G$ is connected, then $\ker(\Ad) = Z(\G)$.
\end{prop}

\begin{rmk}
In particular, in this situation we see that $Z(\G)$ is a  $k$-group
(assuming that $\G$ itself is a $k$-group).
\end{rmk}

\begin{cor}
  \label{cor:adjoint-type}
  Let $\G$ be a connected $k$-group that has trivial center. Then $\G$
  is $k$-isomorphic to $\Ad(\G)$. 
\end{cor}
\begin{proof}
  It is clear that as an abstract group $\Ad(\G)$ is isomorphic to $\G$,
  and in particular the surjective $k$-morphism $\Ad: \G \to \Ad(\G)$ is invertible.
  Under our assumption $\mathrm{char}(k) = 0$, it turns out that the inverse
  $\Ad^{-1}$ is then necessarily a $k$-morphism as well (cf.
  \cite[Ex.~5.3.5]{Spring98}), so that $\G$ is $k$-isomorphic to
  $\Ad(\G)$.
\end{proof}

\begin{definition}[Adjoint group]
  A connected group $\G$ is called \emph{adjoint} if its kernel $Z(\G)$
  is trivial.
\end{definition}

\section{Unipotent elements}
\label{sec:unipotents}

Let $\V$ be a vector space over $\C$, of finite dimension $n$ . 

\begin{definition}[Nilpotent element]
       An element $X \in \End(\V)$ is called \emph{nilpotent}
       if $X^p = 0$ for some integer $p$.
\end{definition}

\begin{prop}
        For $X \in \End(\V)$, the following are equivalent:
        \begin{itemize}
                \item $X$ is nilpotent.
        \item All eigenvalues of $X$ are $0$. 
        \item $\tr(X^j) = 0$ for $j = 1,\dots, n$.
        \end{itemize}
        \label{prop:nilpotents}
\end{prop}

\begin{definition}[Unipotent element]
       An element $g \in \GL(\V)$ is called unipotent if $g - 1$ is nilpotent.
       Equivalently: $g$ is unipotent if all its eigenvalues (in $\C$) equal $1$. 
\end{definition}

Let $\G$ be an algebraic $k$-group for some $k \subset \R$. 
Since $\G \subset \GL(\V_\C)$ and $\g \subset \End(\V_\C)$ for some vector space
$\V$, we may speak of unipotent element in $\G$ or nilpotent elements in $\g$. 

\begin{prop}
        \label{prop:nilp-implies-unip}
       If $\g_k$ possesses a nonzero nilpotent element, then $\G(k)$ has a
       nontrivial unipotent element. 
\end{prop}
\begin{proof}
        If $\lambda_i$ are the eigenvalues of $A \in \End(\V_\R)$ then $e^{\lambda_i}$ 
        are the eigenvalues of its exponential $\exp(A)$. Thus for $X \in \g$
        nilpotent, we have that $g = \exp(X) \in \G(\R)$ is unipotent. But the
        nilpotence also implies that $\exp(X)$ is a finite sum (in
        $\End(\V_k)$), so that   $g \in \G(\R) \cap \End(\V_k) = \G(k)$.

        It remains to prove that $g \neq 1$ for $X$ nonzero. Let $p \ge 1$ be the
        least integer such that $X^p \neq 0$, and suppose that $\exp(X) = 1$.
        For any integer $n$ we can use the formula $\exp(n X) = \exp(X)^n$, so
        that we have the following equality:
        \begin{align*}
                1 + nX + \frac{n^2 X^2}{2!} + \dots + \frac{n^p X^p}{p!} &= 1.
        \end{align*}
        By dividing by $n^p$ and letting $n \to \infty$, we obtain $X^p = 0$.
        This contradicts the choice of $p$. 
\end{proof}

\section{Semisimple groups}

Recall that an abstract group $G$ is called simple if it has no proper normal
subgroups (i.e., other than  $\left\{ 1 \right\}$ and $G$). It is useful to adapt
this definition in the context of algebraic groups.
\begin{definition}[Simple algebraic group]
  \label{def:simple-group}
  An algebraic group $\G$ is called \emph{simple} if it is not
  commutative and has no closed connected normal subgroup apart from $\left\{
  1 \right\}$ and $\G$. 
\end{definition}

\begin{rmk}
 Very often one says ``almost simple'' in this context, to distinguish
 from the notion for abstract groups. Note that if $\G$ is simple as an
 abstract group, then it is simple in the sense of
 definition~\ref{def:simple-group}. Of course the converse does not
 hold.
\end{rmk}

\begin{rmk}
 Excluding commutative groups in the definition excludes $\G_a$ and
 $\G_m$, which both have no proper connected normal subgroups. 
\end{rmk}

Note that in the definition of a simple algebraic group one really exclude the
existence of closed normal subgroup, and not only $k$-closed subgroup (i.e.,
$k$-subgroup). The definition can be weakened to the following.

\begin{definition}[$k$-simple group]
       A noncommutative $k$-group $\G$ is called $k$-simple if it has no normal
       connected $k$-subgroup other than $\left\{ 1 \right\}$ and $\G$.
\end{definition}

\begin{rmk}
       Then ``simple'' implies ``$k$-simple'' (but not conversely). 
\end{rmk}

\begin{example}
 The center of $\GL(\V)$ is isomorphic to  $\Gm$ (seen as the 
 subgroup of ``scalar'' diagonal matrices). Since $\Gm$ is connected,
 it follows that $\GL(\V)$ is not a simple algebraic group (nor it is $k$-simple). 
\end{example}

\begin{example}
 The algebraic group $\SL(\V)$ is simple. One has $Z(\SL(\V)) \cong
 \mu_n$, the group of $n$-th roots of unity in $K$ with $n= \dim(\V)$.
 Thus for $n>1$,  $\SL(\V)$ is not simple as an abstract group.  
 Note that $\mu_n$ is finite of order $n$, and thus not connected for $n>1$.
\end{example}

\begin{prop}
        If $\G$ is simple, then the center $Z(\G)$ is finite.
\end{prop}
\begin{proof}
  We explain the idea.
  The center $Z(\G)$ is a closed normal subgroup, but not necessarily
  connected. But the connected component  $Z(\G)^\circ$ containing $1$ 
  is still closed and normal, and by definition connected (cf. \cite[\S 7.3]{Humph75}
  for more information about the connected component of identity). Since $\G$ is
  not commutative we have necessarily $Z(\G)^\circ = \left\{ 1 \right\}$. It
  is a general fact that the connected component of $1$ has finite index in
  the group; in this case: $[Z(\G) : \left\{ 1 \right\}] < \infty$, so
  that $Z(\G)$ is finite.
\end{proof}

In characteristic zero, there is a strong relation between closed subgroups of
$\G$ and Lie  subalgebras of $\g$. This is the content of the following theorem.
See \cite[\S 13.1 and \S 13.3]{Humph75}

\begin{theorem}
        \label{thm:subgroups-Lie-sbalg}
        Let $\G$ be a connected $k$-group (with $\mathrm{char}(k) = 0$).
        There is a one-to-one correspondence
        between the closed connected subgroups of $\G$ and the Lie subalgebras of $\g$. 
        Under this correspondence, normal subgroups corresponds to ideals in
        $\g$. 
\end{theorem}

\begin{rmk}
        By definition, $I \subset \g$ is an ideal if $[I, \g] \subset I$.
\end{rmk}

This theorem suggests that the assertion ``$\G$ simple'' can be checked at the
level of the Lie algebra. This can be still tricky, but tends to be a more
elementary question.

\begin{cor}
        A noncommutative algebraic group $\G$ is simple if and only
        if $\g$ is a simple Lie algebra, i.e., it has no ideals apart from
        $\left\{ 0 \right\}$ and $\g$.
\end{cor}

It should be noted that the Lie algebra does not completely determine
the algebraic group. For instance, one has the following.  

\begin{lemma}
  \label{lemma:Lie-isogeny}
  Let $\phi: \G \to \H$ be a surjective morphism of group with finite kernel.
  Then the Lie algebra $\h = \Lie(\H)$ is isomorphic to $\g$.
\end{lemma}
\begin{proof}
        Let $\mathfrak{k} \subset \g$ be the Lie algebra $\Lie(\ker(\phi)$.
       The two following properties are  specific to the
       characteristic zero case (cf. \cite[Chapter V]{Humph75}):
       \begin{itemize}
               \item $\phi$ being surjective, the differential $d\phi_e: \g \to
                       \h$ is surjective;
               \item we have: $\ker(d\phi_e) = \Lie(\ker(\phi))$. 
       \end{itemize}
       For $\ker(\phi)$ finite, we have $\mathfrak{k} = \left\{ 0
       \right\}$, whence $\h \cong \g/\mathfrak{k} = \g$.  
\end{proof}

\begin{rmk}
  Such a map $\phi$ is called an \emph{isogeny}.
\end{rmk}

\begin{cor}
  If $\G$ is simple, then $\Lie(\Ad(\G)) \cong \g$.
\end{cor}

\begin{prop}
  Let $\G$ be a connected $k$-simple algebraic group. Then $\G$ has no
  nontrivial $k$-character.
\end{prop}
\begin{proof}
  Again, we limit ourself to the idea of the proof.
  If $\chi: \G \to \Gm$ is a $k$-character, then $\H = \ker(\chi)$ is a
  normal $k$-subgroup, and so is $\H^\circ$ -- its connected component
  of $1$. By definition of ``$k$-simple'', we have either $\H^\circ =
  \G$ (in which case $\chi$ is trivial), or $\H^\circ = \left\{ 1
  \right\}$. In the latter case $\ker(\chi)$ is finite (since $[\H :
  \H^\circ]$ is finite) and by Lemma~\ref{lemma:Lie-isogeny} we have
  that $\g \cong \Lie(\Gm)$, which is only possible if $\G$ (assumed
  connected) is commutative. 
\end{proof}

\vspace{0.3cm}

A Lie algebra $\g$ is called \emph{semisimple} if it is the direct sum $\g =
\g_1 \oplus \dots \oplus \g_s$ of simple Lie algebras. 

\begin{definition}
        An connected algebraic group $\G$ is called \emph{semisimple}
        if its Lie algebra $\g$ is semisimple.
\end{definition}

\begin{rmk}
        In particular, if $\G$ is simple then it is semisimple.
\end{rmk}

\begin{rmk}
        This is an ad hoc definition in the case of characteristic zero. The
        general definition is the following: $\G$ is semisimple if it has no
        nontrivial connected normal solvable closed subgroup.
\end{rmk}

\begin{example}
       If $\G_1,\dots, \G_s$ are simple algebraic group, then their product $\G_1
       \times \dots \times \G_s$ is semisimple, with Lie algebra $\g = \g_1
       \oplus \dots \oplus \g_s$ (where $\g_i = \Lie(\G_i)$). 
\end{example}

We state here without proof the following facts that should clarify the
structure of semisimple algebraic groups. 
The first three properties are an extension of the situation for simple
groups, which has been discussed above. The last point follows from
\cite[Theorem 27.5]{Humph75}.
Let $\G$ be a connected semisimple $k$-group.

\begin{itemize}
        \item The center $Z(\G)$ is finite.
        \item The Lie algebra of $\Ad(\G)$ is isomorphic to $\g$. 
                In particular, $\Ad(\G)$ is semisimple (and adjoint).
        \item  $\G$ has no nontrivial character (in
                particular no nontrivial $k$-character). 
        \item If $\G$ is adjoint, then $\G$ is $k$-isomorphic to a product $\G =
                \G_1 \times \dots \times \G_s$, where each $\G_i$ is simple
                adjoint. 
\end{itemize}

%% file: uniform-lattices.tex
\chapter{Uniform arithmetic lattices}
\label{ch:uniform-arithm-lattices}

Unless stated otherwise, the algebraic groups considered in this chapter are
connected. 

\section{The compactness criterion}
\label{sec:compact-criterion}

The goal of this chapter is to prove the following result.

\begin{theorem}
  \label{thm:compactness-1}
  Let $\G$ be a semisimple adjoint $\Q$-group.  If $\G(\Q)$ does not contain any nontrivial unipotent 
  element, then $\G(\R)/\G_\Z$ is compact. 
\end{theorem}
\begin{rmk}
  This proves that for such a $\G$, any arithmetic subgroup in $\G(\Q)$ is 
  a lattice in $\G(\R)$. This fact readily follows from Theorem~\ref{thm:Borel-HC} 
  (which we will not proved during this course). 
\end{rmk}

Several improvements to this theorem are possible. First, it is possible to drop
the assumption that $\G$ is adjoint, and work with a more general semisimple
$\Q$-group. Moreover, the converse of the theorem also holds: if $\G(\Z)$ is a
uniform lattice in $\G(\R)$ then $\G(\Q)$ cannot contain any unipotent element.

One can show that for a semisimple $\G$ the absence of unipotent
elements in $\G(\Q)$ is equivalent to saying that there is no nontrivial
$\Q$-morphism $\Gm \to \G$ (such a morphism is called a $\Q$-cocharacter). 
In this case, one says that $\G$ is \emph{$\Q$-anisotropic} (in contrast to
$\Q$-isotropic). This last condition, equivalent to the absence of rational
unipotent elements in the semismple case, turns out to be the right condition to
state the general result. The following theorem was conjectured by Godement. Its
proof  is done  by reducing the problem to the statement of 
Theorem~\ref{thm:compactness-1} (and its converse). The proof appear for
instance in \cite{Bor69}.

\begin{theorem}[Godement's compactness criterion]
  \label{thm:Godement}
 Let $\G$ be a $\Q$-group without nontrivial $\Q$-character. Then the
 lattice $\G_\Z \subset \G(\R)$ is uniform if and only if $\G$ is $\Q$-anisotropic. 
\end{theorem}

Together with the theorem of Borel and Harish-Chandra, this result
completely solves the problem of determining when the arithmetic subgroups in
a $\Q$-group $\G$ are lattices or uniform lattices in $\G(\R)$.

\section{Arithmetic groups and representations}

\begin{lemma}
  \label{lemma:congr-stab-lattice}
  Let $\G$ be a $\Q$-group and $\V_\Q$ be a $\Q$-vector space with a
  $\Q$-representation $\G \to \GL(\V)$. Let $\Gamma \subset \G(\Q)$ be 
  an  arithmetic subgroup and $L \subset \V_\Q$ be a $\Z$-lattice. Then there
  exists a subgroup of finite index $\Gamma_0 \subset \Gamma$ that stabilizes
  $L$: for any $g \in \Gamma_0$ one has $g L = L$.
\end{lemma}
\begin{proof}
       Let us assume that $\G \subset \GL(\W)$. For simplicity we will assume
       that $\G$ has no $\Q$-character, but the same proof works by realizing
       $\G$ as a closed subgroup of some $\SL(\W)$.

       By definition, there is a
       $\Z$-lattice $L' \subset \W_\Q$ such that $\Delta = \G_L \cap \Gamma$ has
       finite index in $\Gamma$. Let us write $\Gamma = g_1 \Delta \cup \dots
       \cup g_r \Delta$. Then the lattice $L'' = \sum_{i=1}^s g_i L'$ is a
       $\Z$-lattice in $\W_\Q$ invariant by $\Gamma$. By choosing a $\Z$-basis
       of $L''$, we obtain an embedding $\Gamma \subset \GL_n(\Z)$ (where $n =
       \dim(\W)$ and with the identification $\GL(\W_\Q) = \GL_n(\Q)$).

       Let us now consider a $\Q$-representation $\phi: \G \to \GL(\V)$, and $L
       \subset \V_\Q$ a $\Z$-lattice. By choosing a basis of $L$, 
       we can identify $\GL(\V_\Q) = \GL_s(\Q)$.
       Then $\phi$ can be written by matrix components $\phi = (\phi_{ij})$,
       where for $g \in \GL_n(\overline{\Q})$,  each $\phi_{ij}(g)$ is given by a  polynomial 
       in terms of $g_{11}, \dots, g_{nn}$ with coefficients in $\Q$
       (note that $\det(g) = 1$, since we assume that $\G$ has no $\Q$-character).
       Let $m$ be the least common multiple to
       all denominators in the coefficients of all the $\phi_{ij}$'s. Then the
       polynomials $f_{ij} = m \phi_{ij}$ have coefficients in $\Z$. 
       Let $\Gamma_0 \subset \Gamma$ be the congruence subgroup  
       \begin{align*}
               \Gamma_0 &= \left\{ g \in \Gamma \;|\; g \equiv 1 \mod m
               \right\},
       \end{align*}
       which has finite index in $\Gamma$. Now, for $g \in
       \Gamma_0$ we have $f_{ij}(g) - f_{ij}(e) \in m \Z$ for any $(i,j)$, so that 
       $\phi_{ij}(g) \in \Z$. It follows that $\phi(\Gamma) \subset \GL_s(\Z)$,
       which means that $\Gamma$ stabilizes the $\Z$-lattice $L$. 
\end{proof}

\begin{lemma}
        \label{lemma:arithm-stab-lattice}
        Let $\G \to \GL(\V)$ be a $\Q$-representation, and $\Gamma \subset
        \G(\Q)$ an arithmetic subgroup.
        There exists a $\Z$-lattice $L \subset \V_\Q$ stabilized by $\Gamma$.
\end{lemma}
\begin{proof}
        Let $L_0 \subset \V_\Q$ be any $\Z$-lattice, and (by the preceding
        lemma) let $\Gamma_0 \subset \Gamma$ a subgroup of finite index
        stabilizing $L_0$. For $\Gamma = g_1 \Gamma_0 \cup \dots \cup g_n
        \Gamma_0$, let us define $L = \sum_{i=1}^n g_i L_0$. Then $L$ is a
        $\Z$-lattice stabilized by $\Gamma$.
\end{proof}

\begin{rmk}
        To say that $\Gamma$ stabilizes a lattice is not the same as saying that
        $\Gamma$ is the full stabilizer of the lattice.
\end{rmk}

\begin{rmk}
        Note that in Lemma~\ref{lemma:arithm-stab-lattice}, 
        the lattice $L$ stabilized by $\Gamma$ can be chosen to contains a given
        $v \in \V_\Q$ (indeed, to contain a given lattice $L_0$). 
\end{rmk}

\begin{prop}
        Let $\phi: \G \to \H$ be a $\Q$-isomorphism of $\Q$-groups. 
        If $\Gamma \subset \G(\Q)$ is arithmetic, then $\phi(\Gamma)$ is
        arithmetic in $\H(\Q)$. 
\end{prop}

\begin{proof}
        Suppose that $\Gamma \subset \G(\Q)$ is arithmetic.
        By Lemma~\ref{lemma:arithm-stab-lattice} the image $\phi(\Gamma)$
        stabilizes a lattice $L$, so that $\phi(\Gamma) \subset \H_L$. Applying
        the same reasoning with $\phi^{-1}$, there exists $L'$ such that 
        \begin{align*}
                \Gamma \subset \phi^{-1}(\H_L) \subset \G_{L'}.
        \end{align*}
        Then the index $[\H_L : \phi(\Gamma)] = [\phi^{-1}(\H_L): \Gamma]$ is
        bounded by $[\G_{L'}:\Gamma]$, and thus finite. This shows that
        $\phi(\Gamma)$ is arithmetic.
\end{proof}

Since $\phi$ induces an homeomorphism $\G(\R)/\Gamma \cong \H(\R)/\phi(\Gamma)$,
we obtain the following.

\begin{cor}
        \label{cor:isom-preserves-lattices}
        $\G_\Z \subset \G(\R)$ is a (uniform) lattice if and only $\H_\Z \subset
        \H(\R)$ is a (uniform) lattice. 
\end{cor}

In particular, for  $\G$ semisimple and adjoint we can replace the study of
$\G_\Z$ by the study of $\Ad(\G)_\Z$ in $\Ad(\G)(\R)$.

\section{A proper embedding}

Let $X = \GL_n(\R)/\GL_n(\Z)$, that is, $X$ can be identified with the space of
of Euclidean lattices in $\R^n$ (see Chapter $1$). Using Mahler compactness
criterion, we can deduce in a straightforward way that $\SL_n(\R)/\SL_n(\Z)$ is
not compact: if it were, so would be its image in $X$, which is not the case (the
systole of lattices parametrized by $\SL_n(\R)$ is arbitrarily small). 

Let $\G \subset \GL(\V)$ be a $\Q$-group.
To prove Theorem~\ref{thm:compactness-1}, the idea is also to reduce the problem
to an application of Mahler's compactness criterion on the embedding
$\G(\R)/\G_L \hookrightarrow X$, where $X$ is identified with
$\GL(\V_\R)/\GL(L)$ for some $\Z$-lattice $L \subset \V_\Q$. But \emph{a priori}
compactness of the image in $X$ does not implies compactness of $\G(\R)/\G_L$:
we have to show that the embedding has good properties (namely, that it is
\emph{proper}). This is the goal of this section. 

Let us fix a basis of $L \subset \V_\Q$, so that $\GL(\V_\Q)$ is identified with
$\GL_n(\Q)$, and $\G_\Z = \G_L$.

\begin{lemma}
        \label{lemma:proper-embedding}
        Suppose that $\G$ has no nontrivial $\Q$-character.
        Then the product $\G(\R)\cdot\GL_n(\Z)$ is closed in $\GL_n(\R)$.
\end{lemma}

\begin{proof}
     Let $g_n \cdot \gamma_n$ that converges to $h \in \GL_n(\R)$, with $g_n \in
     \G(\R)$ and $\gamma_n \in \GL_n(\Z)$. We have to show that $h \in \G(\R)
     \GL_n(\Z)$. By the corollary to Chevalley's theorem
     (Corollary~\ref{cor:Chevalley}), there exists a $\Q$-representation of
     $\GL_n$ on some $\V$ such that $\G$ is the stabilizer of some $v \in
     \V_\Q$. Then $g \in \GL_n(\R)$ belongs to $\G(\R)$ if and only if $g v =
     v$. It follows that $\gamma_n^{-1} g_n^{-1} v = \gamma_n^{-1}v$ converges
     to $h^{-1} v$ in $\V_\R$.

     Using Lemma~\ref{lemma:arithm-stab-lattice} we can choose a $\Z$-lattice $L
     \subset \V_\Q$ containing $v$ and stabilized by $\GL_n(\Z)$.
     Thus $\gamma_n^{-1} v \in L$ form a discrete set in $\V_\R$, proving that
     $\gamma_n^{-1} v = h^{-1} v$ for $n$ large enough. This implies that $h
     \gamma_n^{-1}$ stabilizes $v$ for $n$ large enough, so that $h \in \G(\R)
     \GL_n(\Z)$.
\end{proof}

\begin{prop}
        \label{prop:proper-embedding} 
        Suppose that $\G$ has no nontrivial $\Q$-character.
        The injection $\G(\R)/\G_\Z \to X$ is a homeomorphism onto a closed subset of $X$.
        In particular, $\G(\R)/\G_\Z$ is compact if and only if its image in $X$
        satisfies to the conditions of Mahler's compactness criterion.
\end{prop}
\begin{proof}
       By definition of the quotient topology, the image $\G(\R)/\G_\Z$ in $X$
       is closed if and only if $\G(\R) \GL_n(\Z)$ is closed in $\GL_n(\R)$,
       which was proved in the preceding lemma. 
       To see that the injection is a homeomorphism we need to show that 
       if $S \G_\Z$ is closed in $\G(\R)$ then $S \GL_n(\Z)$ is closed in
       $\GL_n(\R)$ (passing to the quotient this shows the continuity of the
       inverse of the injection). But if $S \G_\Z$ is closed, we may write it as
       $\G(\R) \setminus U$ for some $U \subset \GL_n(\R)$ open. Then $S
       \GL_n(\Z)$ = $\G(\R) \GL_n(\Z) \setminus U \GL_n(\Z)$, which is closed.

       For the last assertion, denote by $A$ the image of $\G(\R)/\G_\Z$ in $X$.
       If $A$ satisfies the conditions of Mahler's criterion, then $A =
       \overline{A}$ is compact, and thus $\G(\R)/\G_\Z$ is compact (since the
       injection is a homeomorphism). 
\end{proof}

\section{Proof of the theorem}

Let $\G$ be semisimple and adjoint, so that $\G$ has no nontrivial
$\Q$-character and $\G$ is $\Q$-isomorphic to $\Ad(\G)$. Suppose that 
$\G \subset \GL(\V)$,  where $\V_\Q$ is a $\Q$-vector space of dimension $N$.
Let $L \subset \V_\Q$ be a $\Z$-lattice; fixing a basis of $L$ we have an
identification $\End(\V_\Q) = \Q^{N\times N}$, and $\GL(\V_\Q) = \GL_N(\Q)$.

Let $\g \subset \End(\V)$ be the Lie algebra of $\G$. Then $\g_\Q$ is a
$\Q$-subspace of $\End(\V_\Q)$, and it follows that the intersection $\g_\Z
= \End(L) \cap \g_\Q$ is a $\Z$-lattice in $\g_\Q$. By construction, $\g_\Z \subset
\Z^{N \times N}$. $\G$ acts on $\g$ via the adjoint representation: $\Ad(\G)
\subset \GL(\g)$; but note that since $\G$ has no nontrivial $\Q$-character, we
actually have $\Ad(\G) \subset \SL(\g)$. In particular, the adjoint
representation preserves the Euclidean volume on $\g_\R$.
We will denote by $\Ad(\G)_\Z$ the stabilizer of $\g_\Z$ in $\Ad(\G)(\Q)$. 

Suppose that $\G(\R)/\G_L$ is not compact. Since $\G \cong \Ad(\G)$, by
Corollary~\ref{cor:isom-preserves-lattices} this is the case exactly when
$\Ad(\G)(\R)/\Ad(\G)_\Z$ is not compact. To proves
Theorem~\ref{thm:compactness-1}, we need to show that in this case $\G(\Q)$
possesses a nontrivial unipotent element. By
Proposition~\ref{prop:nilp-implies-unip} it suffices to show that $\g_\Q$ has a
nonzero nilpotent element.

Now we apply Proposition~\ref{prop:proper-embedding}: if
$\Ad(\G)(\R)/\Ad(\G)_\Z$ is not compact, then the set of lattices in $\g_\R$
given by 
\begin{align*}
        A &= \left\{ \Ad(g) \cdot \g_\Z \;|\; g \in \G(\R) \right\}
\end{align*}
does not satisfy the conditions of Mahler's criterion. But all these lattices
have same covolume, so that the systole gets arbitrarily small in $A$. 
This means that there exists $g_n \in \G(\R)$ and $X_n \in \g_\Z \setminus
\left\{ 0 \right\}$ such that
$\Ad(g_n) X_n = g_n X_n g_n^{-1}$ converges to $0 \in \g_\R$.
But this implies that for $j = 1,\dots,N$ the trace $\tr(X_n^j)$ converges to
$0$. Since $X_n$ has coefficients in $\Z$ we have that for $n$ large enough
$\tr(X_n^j) = 0$. By Proposition~\ref{prop:nilpotents}, such an $X_n$ must be
nilpotent. This concludes the proof.

%% file: over-nb-fields.tex
\chapter{Arithmetic groups over number fields}

The construction of a $\Q$-group $\G$ such that $\G(\Q)$ does not contain nontrivial
unipotent element is not so obvious. In order to consider such groups --
and thus produce uniform arithmetic lattices -- we will consider
arithmetic groups in connection with algebraic number theory. 

\section{Number fields}

We survey some features of algebraic number theory. We refer to
\cite{SteTall} or \cite{Neuk99} for an introduction to the subject.

\begin{definition}
We will call a subfield $k \subset \C$ a \emph{number field} if seen as
a $\Q$-vector space it has finite dimension: $\dim_\Q k < \infty$.
This dimension is called the \emph{degree} of $k$, and is denoted by
$[k:\Q]$.
\end{definition}

By choosing a basis, it is clear that any number field $k$ can be
written as $k = \Q(a_1,\dots, a_n)$ for some $a_i \in \C$. Here
$\Q(a_1,\dots, a_n)$ stands for the smallest field generated by the
$a_i$, and it can be concretely described as the rational expressions in
the $a_i$'s with coefficients in $\Q$. Moreover, in this description all
$a_i$ are algebraic numbers (otherwise the dimension would be infinite).
It will be useful to have the following -- stronger -- description at hand.

\begin{prop}
  If $k$ is a number field, then $k = \Q(a)$ for some algebraic number
  $a \in \overline{\Q}$.
\end{prop}

Recall that an element $a \in \overline{\Q}$ is an algebraic \emph{integer} 
if it is the zero of a \emph{monic}\footnote{that is, with leading coefficient $=1$.}
polynomial with coefficient in $\Z$. It is a (nontrivial) fact that the set
$\mathbf{A} \subset \overline{\Q}$ 
algebraic integers forms a subring. For $k$ a number field, we denote by $\O = k
\cap \mathbf{A}$ its \emph{ring of algebraic integers}.

\begin{example}
        For $k = \Q$, one has $\O = \Z$. Trivially, this is the only number field of degree $1$.
\end{example}

\begin{theorem}
        \label{thm:integers-order}
        The ring of integer $\O$ of a number field $k$ is a (free) $\Z$-module 
        of rank $d = [k:\Q]$. 
\end{theorem}

\begin{rmk}
        Such a subring $\O \subset k$ (containing $1$) that is a $\Z$-lattice of rank $d$ is
        called an \emph{order} of $k$ (and this explain the notation). It turns
        out that the ring of integers is the maximal order in $k$ (it contains
        all other orders of~$k$).
\end{rmk}

A number field $k$ is called \emph{quadratic} if $[k :\Q] = 2$. Those are the
fields of the form $k = \Q(\sqrt{m})$, where  $m \in \Q^\times$ is not a square.

\begin{example}
        For $k = \Q(\sqrt{2})$, one has $\O = \Z[\sqrt{2}]$.      
\end{example}

\begin{example}
        For $k = \Q(i)$ with $i^2 = -1$, one has $\O = \Z[i]$ (which is called
        the ring of \emph{Gaussian integers}).
\end{example}

\begin{example}
        For $k = \Q(\sqrt{5})$, one has $\O = \Z[\frac{1+\sqrt{5}}{2}]$.      
\end{example}

\vspace{.3cm}

By a \emph{monomorphism} $\sigma: k \to \C$ we mean a homomorphism of
ring, sending $1 \mapsto 1$. In particular $\sigma$ must be injective, and for
$x \in \Q$ we have $\sigma(x) = x$.
As a trivial example, there is always the embedding $k \subset \C$ given by the
identity map.  If $f(t) \in \Q[t]$ is a polynomial and $a \in k$, then
$f(\sigma(a)) = f(a)$. 

\begin{example}
        In the example of a quadratic field $k = \Q(\sqrt{m})$, beside the
        identity map given  $\sqrt{m}  \mapsto \sqrt{m}$, there is the
        homomorphism determined by $\sqrt{m} \mapsto -\sqrt{m}$.
\end{example}

\begin{prop}
  Let $k$ be a number field of degree $d$. There exist exactly $d$
  distinct monomorphisms $k \to \C$.
\end{prop}

\begin{proof}
        Let $k = \Q(a)$, and let $f \in \Q[t] $ be the minimal polynomial of
        $a$. The idea is that each monomorphism $\sigma : k \to \C$ is
        determined by $a \mapsto a'$, where $a'$ is a \emph{conjugate} of $a$,
        that is, a zero of $f$. It turns out that $\deg(f) = d$, and all its
        zeros are distinct.  See \cite[Theorem 2.4]{SteTall} for details.
\end{proof}

We will say that the monomorphism $\sigma$ is \emph{real} if $\sigma(k) \subset
\R$, and \emph{complex} otherwise. If $\sigma$ is complex, then its
complex conjugate $\overline{\sigma}$ is another monomorphism, so that complex
monomorphisms always come in pairs. If $r_1$ denotes the number of real
monomorphisms, then $d = r_1 + 2r_2$ where $r_2$ is the number of pairs
of complex monomorphisms. We call $(r_1, r_2)$ the {\em signature} of $k$.

\begin{example}
        Let $k = \Q(\sqrt[3]{2})$, which is a \emph{cubic} field (i.e., $[k:\Q] =
        3$). Then $k \subset \R$, and the nontrivial monomorphisms are given by
        $\sigma: \sqrt[3]{2} \mapsto e^{2\pi i/3} \sqrt[3]{2}$ and its complex
        conjugate $\overline{\sigma}$. Thus, $k$ has signature $(1,1)$.
\end{example}

\section{Number fields and discrete groups}

The $\Z$-module $\O$ can be seen as a (Euclidean) lattice in a suitable vector
space. The construction is the following. For each monomorphism $\sigma: k \to
\C$ we set  $k_\sigma = \R$ if $\sigma$ is real, and $k_\sigma = \C$ if $\sigma$
is complex. Then we define 
\begin{align}
        k_\R = \prod_\sigma k_\sigma,
        \label{eq:k_R}
\end{align}
where the product runs over each real monomorphism and over one
representative
in each pair of complex monomorphism. Then, as a real vector space $k_\R$ has
same dimension as the degree $[k:\Q]$. 

\begin{example} \ 
        \begin{itemize}
                \item For $k = \Q(\sqrt{2})$, we have $k_\R = \R \times \R$.
                \item For $k = \Q(i)$, we have $k_\R = \C$.
                \item For $k = \Q(\sqrt[3]{2})$, we have $k_\R = \R \times \C$.
        \end{itemize}
\end{example}

In the example $k = \Q(i)$, we readily see that $\O = \Z[i]$ forms a Euclidean
lattice in $\C \cong \R^2$. This generalizes to the generic situation.
For this, we consider the natural injection map $k \to k_\R$ given by $x \mapsto
(\sigma(x))_\sigma$.  For example for $k = \Q(\sqrt{2})$: $a + b \sqrt{2} \mapsto
(a + b \sqrt{2}, a - b \sqrt{2})$, where $a, b \in \Q$.

\begin{prop}
  The $\Z$-module $\O$ is discrete in $k_\R$.
\end{prop}


In some sense this generalizes the discreteness of $\Z$ in $\R$. We can
take advantage of the fact that $\O$ is a subring to construct discrete
groups in locally compact groups. First we mention the following
example.

\begin{example}
  The group $\SL_n(\Z[i])$ is a discrete subgroup of $\SL_n(\C)$.
  This follows directly from the discreteness of $\O = \Z[i]$ in $\C$.
\end{example}

We would like to generalize this construction for an arbitrarily
$k$-group $\G$, with $k$ a number field. To simplify the notation we
will restrict to the case where $k$ is \emph{totally real}, that is, all
monomorphisms of $k$ are real (e.g., $k = \Q(\sqrt{2})$).
It will be convenient to see $\G$ as a $k$-closed
subgroup $\G \subset \GL(\V_\C)$, and in turn identify $\GL(\V_\C) =
\GL_N(\C)$ (by choosing a basis of $\V_k$).
Being $k$-closed, $\G$ can be written as
\begin{align}
  \left\{ g \in \GL_N(\C) \;|\; f(g_{11},\dots, g_{nn}) = 0 \mbox{ for all }
f \in S \right\},
  \label{eq:G-eq}
\end{align}
for some set of polynomials $S \subset k[T_{11},\dots, T_{nn}]$. For any
monomorphism $\sigma: k \to \R$, we define the group $\G^\sigma$ to be   
\begin{align}
  \left\{ g \in \GL_N(\C) \;|\; (^\sigma f)(g_{11},\dots, g_{nn}) = 0 \mbox{ for all }
f \in S \right\},
  \label{eq:G-eq-sigma}
\end{align}
where the polynomial $^\sigma f$ is obtained from $f$ by applying $\sigma$ on each of
its coefficients. Thus $\G^\sigma$ is a $\sigma(k)$-closed subset, which is
actually a group (this is not completely obvious in general, but will be on the
examples we will consider). We then
define $\G(k_\R)$ as the product
\begin{align}
  \G(k_\R) &= \prod_\sigma \G^\sigma(\R),
  \label{eq:}
\end{align}
where the product runs over each monomorphism $\sigma: k \to \R$.

\begin{rmk}
  For $k$ not totally real the construction has to be modified as
  follows: the product takes one representative in the pair $(\sigma,
  \overline{\sigma})$, and for $\sigma$ complex we consider $\G^\sigma(\C)$
  instead of $\G^\sigma(\R)$.
\end{rmk}

\begin{example}
  Consider $\G = \SL_2$ as a $k$-group, for $k = \Q(\sqrt{2})$. Then
  $\G$ is determined by the single polynomial $f = \det - 1$, which
  actually has coefficients in $\Q$. It follows that $^\sigma f = f$ for
  both $\sigma : k \to \R$. Thus, $\G(k_\R) = \SL_2(\R) \times
  \SL_2(\R)$.
\end{example}

Let us fix a $\O$-lattice $L \subset \V_k$, that is, $L$ is a free $\O$-submodule
of rank $= N = \dim(\V_k)$.  Then the stabilizer $\G_L \subset \G(k)$
identifies (by choosing a basis of $L$) as the intersection $\G_L =
\G(k) \cap \GL_N(\O)$, where the latter denotes matrices with
coefficients in $\O$ and determinant in $\O^\times$ (this identifies with the
group of automorphisms of $L$). This justifies the notation $\G_\O = \G_L$.

\begin{prop}
  The natural inclusion $k \subset k_\R$ induces an inclusion $\G(k) \subset
  \G(k_\R)$, for which $\G_\O$ becomes a discrete subgroup of $\G(k_\R)$.
\end{prop}

\begin{proof}
        For $g \in \G(k) \subset \GL_n(k)$, we set $\sigma(g) =
        (\sigma(g_{ij}))_{i,j}$ (apply $\sigma$ on each matrix component). Then
        one checks that $\sigma(g) \in \G^\sigma(\R)$. The inclusion $\G(k)
        \subset \G(k_\R)$ is given by $g \mapsto (\sigma(g))_\sigma$. Since the
        group $\G(k_\R)$ can be seen as a set of matrices with coefficients in
        $k_\R$, the discreteness of $\G_\O \subset \GL_N(\O)$ follows directly
        from the fact that $\O$ is discrete in $k_\R$.
\end{proof}

\begin{example}
  $\Gamma = \SL_2(\Z[\sqrt{2}])$ is discrete in $G = \SL_2(\R) \times \SL_2(\R)$. 
  An obvious discrete subgroup of $G$ would be $\SL_2(\Z) \times
  \SL_2(\Z)$. The group $\Gamma$ is a less trivial construction, in the
  sense that it has dense image in both factors of $G$. One says that
  $\Gamma$ is \emph{irreducible} in $G$.
\end{example}

\begin{definition}
        \label{def:arithm-groups-k}
        A subgroup $\Gamma \subset \G(k)$ that is commensurable with $\G_\O$
        will be called an {\em arithmetic subgroup}, and $k$ is called its {\em field
        of definition}.
\end{definition}

\section{Restriction of scalars}


Let $k \subset \C$ be a number field and $\V_k$ be a vector space of dimension
$n$. We denote by $\V_\Q'$ the set $\V_k$ seen as $\Q$-vector space. If $d =
[k:\Q]$ denotes the degree of $k$, we have that $\V_\Q'$ has dimension $= nd$.
Moreover, if $L \subset \V_k$ is a $\O$-lattice, then we can see it as a
$\Z$-module of $\V'_\Q$, of rank $nd$. We will denote $L'$ this $\Z$-module.

Suppose that $\G \subset \GL(\V)$ is an algebraic $k$-group. There exist a
$\Q$-group $\G' \subset \GL(\V') \cong \GL_{nd}(\overline{\Q})$ for which there
exist the following canonical identifications:
\begin{itemize}
        \item $\G'(\Q) =  \G(k)$;
        \item $\G'(\R) = \G(k_\R)$;
        \item $\G'_{L'} = \G_L$.
\end{itemize}

The group $\G'$ is said to be obtained from $\G$ by \emph{restriction of the
scalars}. In some sense this operation is opposite to the extension of scalars.
The last property above shows the arithmetic subgroup $\G_L = \G_\O$ (defined
over $k$) corresponds to a arithmetic group $\G'_\Z \subset \G'(\Q)$.  This
shows that Definition~\ref{def:arithm-groups-k}, which in appearance generalizes
the notion of arithmetic groups over $\Q$, actually does not produce more
groups but give an alternative description to the construction we already had at
hand.

\begin{rmk}
       The assignment $\G \mapsto \G'$ is a functor from the category of the
       $k$-groups to the category of $\Q$-groups. It is know as \emph{Weil's
       restriction of scalars}, and denoted by $\G' = \mathrm{Res}_{k/\Q}(\G)$. 
\end{rmk}

We refer to \cite[Ch.3 \S  61]{OniVinb-I} and \cite[\S 7.16]{Bor69} for more
information about restriction of scalars.

\begin{example}
  Let $\G = \Gm/ k$ (i.e., seen as a $k$-group) for $k = \Q(\sqrt{m})$
  quadratic with $m > 0$ (so that $k$ is totally real).
  For the following algebraic $\Q$-group:
  \begin{align*}
          \G' &=
  \left\{ 
    \left(
    \begin{array}[h]{cc}
            g_{11} & g_{12}  \\
            g_{21} & g_{22} 
    \end{array}\right) \in \GL_2(\overline{\Q})\;\Big|\; g_{12} = g_{22};\; g_{12} = m g_{21}
      \right\},
  \end{align*}
  we have an isomorphism $\G(k) \cong \G'(\Q)$ given by 
  \begin{align*}
  a + b \sqrt{m} \mapsto   
        \left(
    \begin{array}[h]{cc}
            a & m b  \\
            b & a
    \end{array}\right).
  \end{align*}
  And one checks that the map $\left(  \begin{smallmatrix}
  a & m b \\ b & a
  \end{smallmatrix}\right) \mapsto (a + b \sqrt{m}, a - b \sqrt{m})$
  provides an isomorphism $\G'(\R) \cong \R^\times \times \R^\times = \G(k_\R)$.
  (Note that for $m < 0$ the latter is replaced by $\C^\times$.) If $\O =
  \Z[\sqrt{m}]$ (e.g., for $k = \Q(\sqrt{2})$), then the isomorphism $\G(k)
  \cong \G'(\Q)$ induces an isomorphism $\O^\times = \G_\O \cong \G'_\Z$.
\end{example}

\begin{example}
        For $\G = \SL_n$ over $k = \Q(\sqrt{2})$, the $\Q$-group $\G' \subset
        \GL_4$ obtained by restriction of scalars would then be such that
        $\G'(\Q) = \SL_2(\Q)$ (canonical isomorphism) and $\G'(k_\R) = \SL_2(\R)
        \times \SL_2(\R)$.
\end{example}

Using the property $\G'_{L'}  = \G_L$ permits to reduces question about
arithmetic subgroups in $\G(k)$ to the study of ``classical'' (over $\Q$)
arithmetic subgroups. As a first example, let us mention:

\begin{prop}
       Let $L$ and $M$ be two $\O$-lattices in $\V_k$, and $\G \subset \GL(\V)$
       be a $k$-group. Then $\G_L$ is commensurable to $\G_M$ in $\G(k)$.
\end{prop}
\begin{proof}
        We have $\G_L = \G'_{L'}$, which by
        Proposition~\ref{prop:arithm-commensur} is commensurable to $\G'_{M'} =
        \G_M$.
\end{proof}

We will also need the following properties
(without proof):
\begin{itemize}
        \item If $\G$ is simple, then $\G'$ is semisimple;
        \item an element  $g \in \G(k)$ is unipotent if and only if it is
                unipotent as an element in $\G'(\Q)$
\end{itemize}

\section{Arithmetic lattice over number fields}

From the properties of the restriction of scalars, the following is immediate
from Theorem~\ref{thm:compactness-1}. As mentioned in
Section~\ref{sec:compact-criterion}, the assumption ``adjoint'' can be remove
without affecting the result.

\begin{theorem}
  \label{thm:uniform-over-nb-field}
  Let $\G$ be a simple adjoint $k$-group, for $k$ a number field.
  If $\G(k)$ has no unipotent element, then $\G(k_\R)/\G_\O$ is compact. 
\end{theorem}

\begin{rmk}
        Similarly one can formulate the theorem of Borel and Harish-Chandra, and
        Godement's compactness criterion for $k$-groups. 
        There ``$\Q$-character'' and ``$\Q$-anisotropic'' must be replaced by
        ``$k$-character'' and ``$k$-anisotropic''.
\end{rmk}

\begin{lemma}
        Let $\G$ be an $\R$-group and suppose that $\G(\R)$ is compact.
        Then $\G(\R)$ has no nontrivial unipotent element. 
        \label{lemma:compact-no-unipot}
\end{lemma}
\begin{proof}
        Suppose that  $g \in \G(\R) \subset \GL_n(\C)$ is unipotent and nontrivial. 
        By the Jordan normal form
        theorem, there exists $h \in \GL_n(\C)$ such that 
        \begin{align*}
                h g h^{-1} &= \left(  \begin{array}[h]{ccc} 
                        1 & 1 & * \\
                        0 & 1 & * \\
                        \vdots & 0 & \ddots
                \end{array} \right).
        \end{align*}
        It follows that for $x = h g^n h^{-1}$, we have $x_{12} = n$. Thus the
        group $h \G(\R) h^{-1}$ is not bounded in $\GL_n(\C)$. Since conjugation
        by $h$ gives an homeomorphism, we deduces that $\G(\R)$ is not compact. 
\end{proof}

\begin{cor}
        Let $\G$ be a simple (adjoint) $k$-group, for $k$ a totally real number field.
          If there is $\sigma$ such that $\G^\sigma(\R)$ is compact, then
          $\G_\O$ uniform lattice in $\G(k_\R)$. 
\end{cor}
\begin{proof}
    Since $\G(k) \subset \G^\sigma(\R)$ and the latter is compact, by the
    previous lemma we have that $\G(k)$ has no nontrivial unipotent element.
\end{proof}

\begin{cor}
  If for each $\sigma \neq \mathrm{id}$ we have that $\G^\sigma(\R)$ compact,
  then $\G_\O$ is a uniform lattice in $\G(\R)$.
\end{cor}
\begin{proof}
        We have $$\G(k_\R) = \G(\R) \times \prod_{\sigma \neq \mathrm{id}}
        \G^\sigma(\R),$$
        and thus projection $\phi: \G(k_\R) \to \G(\R)$ onto the first factor
        is a continuous open surjective homomorphism with compact kernel. It follows from
        Proposition~\ref{prop:lattice-compact-kern} that the
        image of $\phi(\G_\O)  = \G_\O \subset \G(\R)$ is a uniform lattice.
\end{proof}
\begin{rmk}
  Indeed we have that $\G_\O$ is an arithmetic lattice in the sense of
  Definition~\ref{def:arithm-lattice}.
\end{rmk}

\begin{example}

        Let $k = \Q(\sqrt{2})$ and $\V_k$ be a $k$-vector space of dimension $n
        +1$. Let $f$ be the quadratic form on $\V_k$ given (for some basis of
        $\V_k$): 
        \begin{align*}
                f(x) &= -\sqrt{2} x_0^2 + x_1^2 + \dots + x_n^2.
        \end{align*}
        The subgroup $\G = \SO(\V, f) \subset \SL(\V)$ of transformations that
        preserves $f$ is a $k$-group. This can be seen by writing $\G$ as the
        set of matrices such that $^t g A g = A$, where $A$ is the diagonal
        matrix $A = (-\sqrt{2}, 1, \dots, 1)$. 

        For the nontrivial monomorphism $\sigma: \sqrt{2} \mapsto - \sqrt{2}$,
        we have that $\G^\sigma = \SO(\V, ^\sigma\!\! f)$, for $^\sigma f(x) =
        \sqrt{2} x_0^2 + x_1^2 + \dots x_n^2$. In particular, $\G^\sigma(\R) =
        \SO(\V_\R, f)$ is compact (and isomorphic to the ``standard'' special
        orthogonal group $\SO(n+1)$).  

        It turns out that for $n \ge 2$ (i.e., $\dim(\V) \ge 3$) the 
        algebraic group  $\G = \SO(\V, f)$ is simple, 
        and adjoint if $n$ is even. By the corollary, we have that $\G(\O)$ is a
        uniform lattice in $\G(\R) = \SO(\V_\R, f) \cong \SO(1, n)$ (the special
        orthogonal group preserving the standard Lorentzian quadratic form).
\end{example}


\begin{rmk}
        By letting the number field $k$ and/or the quadratic form $f$ vary, we
        get many more examples of uniform arithmetic lattices in $\SO(1, n)$, or
        more generally in products of orthogonal groups.
\end{rmk}
